\documentclass[12pt]{amsart}
\usepackage{amssymb}
\usepackage[shortlabels]{enumitem}
\usepackage[all]{xy}
\usepackage{xcolor}
\usepackage[margin=1in]{geometry} 
\usepackage{hyperref}
\usepackage{mathrsfs}
\usepackage{eucal}
\usepackage{contour}
\usepackage{shuffle}
\usepackage[normalem]{ulem}
\hypersetup{bookmarksdepth=2}
\hypersetup{colorlinks=true}
\hypersetup{linkcolor=blue}
\hypersetup{citecolor=blue}
\hypersetup{urlcolor=blue}

\makeatletter
\@addtoreset{equation}{section}
\makeatother

\numberwithin{equation}{section}
\newtheorem{theorem}[equation]{Theorem}
\newtheorem{proposition}[equation]{Proposition}
\newtheorem{lemma}[equation]{Lemma}
\newtheorem{corollary}[equation]{Corollary}

\newtheorem{maintheorem}{Theorem}

\theoremstyle{definition}
\newtheorem{remark}[equation]{Remark}
\newtheorem{example}[equation]{Example}

\newtheorem{definition}[equation]{Definition}

\newcommand{\cA}{\mathcal{A}}

\newcommand{\cC}{\mathcal{C}}

\newcommand{\cF}{\mathcal{F}}

\newcommand{\rK}{\mathrm{K}}

\newcommand{\rN}{\mathrm{N}}

\newcommand{\bR}{\mathbf{R}}

\newcommand{\bS}{\mathbf{S}}

\newcommand{\fS}{\mathfrak{S}}

\newcommand{\bT}{\mathbf{T}}

\newcommand{\fT}{\mathfrak{T}}

\newcommand{\fn}{\mathfrak{n}}

\newcommand{\fp}{\mathfrak{p}}

\newcommand{\fq}{\mathfrak{q}}
\newcommand{\arxiv}[1]{\href{http://arxiv.org/abs/#1}{{\tiny\tt arXiv:#1}}}

\newcommand{\DOI}[1]{\href{http://doi.org/#1}{\color{purple}{\tiny\tt DOI:#1}}}
\newcommand{\defn}[1]{\emph{#1}}

\contourlength{1pt}

\contourlength{0.8pt}
\newcommand{\myuline}[1]{%
  \uline{\phantom{#1}}%
  \llap{\contour{white}{#1}}%
}
\DeclareMathOperator{\uRep}{\text{\myuline{\rm Rep}}}
\DeclareMathOperator{\uPerm}{\ul{Perm}}

\let\ul\underline

\renewcommand{\phi}{\varphi}

\DeclareMathOperator{\End}{End}
\DeclareMathOperator{\Aut}{Aut}
\DeclareMathOperator{\Hom}{Hom}
\DeclareMathOperator{\uHom}{\ul{Hom}}
\DeclareMathOperator{\Rep}{Rep}

\DeclareMathOperator{\Spec}{Spec}
\DeclareMathOperator{\Et}{Et}
\DeclareMathOperator{\im}{im}

\DeclareMathOperator{\Ind}{Ind}
\DeclareMathOperator{\Res}{Res}
\DeclareMathOperator{\udim}{\ul{dim}}
\newcommand{\id}{\mathrm{id}}

\renewcommand{\Vec}{\mathrm{Vec}}
\newcommand{\GL}{\mathbf{GL}}
\newcommand{\op}{\mathrm{op}}

\newcommand{\bone}{\mathbf{1}}

\newcommand{\GG}{\mathbb{G}}

\let\wa\bullet
\let\wb\circ

\newcommand{\bb}{{\bullet}}
\newcommand{\ww}{{\circ}}

\title[Simple algebras in the Delannoy category]{Classification of simple commutative algebras \\ in the Delannoy category}

\author{Pavel Etingof}
\address{Department of Mathematics, MIT, Cambridge, MA 02139, USA}
\email{\href{mailto:etingof@math.mit.edu}{etingof@math.mit.edu}}
\urladdr{\url{https://math.mit.edu/~etingof/}}
\thanks{PE was partially supported by NSF grants DMS-2001318 and DMS-2502467}

\author{Andrew Snowden}
\address{Department of Mathematics, University of Michigan, Ann Arbor, MI, USA}
\email{\href{mailto:asnowden@umich.edu}{asnowden@umich.edu}}
\urladdr{\url{http://www-personal.umich.edu/~asnowden/}}
\thanks{AS was supported by NSF grant DMS-2301871.}

\date{June 25, 2026}

\begin{document}

\begin{abstract}
The Delannoy category is an interesting pre-Tannakian category associated to the oligomorphic group $\GG$ of automorphisms of the totally ordered set $(\bR, <)$. By construction, it admits some obvious simple commutative algebras, corresponding to certain transitive $\GG$-sets. We show that these account for all of the simple commutative algebras in the Delannoy category. Previous results of this kind have been limited to interpolation categories; since the Delannoy category cannot be obtained by interpolation, new methods are required.
\end{abstract}

\maketitle
\tableofcontents

\section{Introduction}

Pre-Tannakian categories are a natural class of tensor categories generalizing representation categories of algebraic groups (see \S \ref{ss:tencat} for the definition). One of the first interesting examples is Deligne's category $\uRep(\fS_t)$, which is obtained by ``interpolating'' the representation categories of finite symmetric groups \cite{Deligne}. More recently, a number of quite different examples have been constructed using the oligomorphic group approach introduced in \cite{repst}. Of these, the most prominent is the Delannoy category, studied in depth in \cite{line}.

Each pre-Tannakian category provides a whole ``world'' in which one can do algebra: one can consider commutative algebras, Azumaya algebras, Lie algebras, etc., internal to such a category. Classifying various kinds of algebraic objects like these is a natural problem. To date, progress has been limited to Deligne's category $\uRep(\fS_t)$ \cite{HarmanKalinov, Sciarappa}; here, one can use the interpolation description to essentially reduce the problem to questions about finite symmetric groups\footnote{It is possible to study other interpolation categories in the same manner, e.g., one can show that the only simple commutative algebra in $\uRep(\GL_t)$ is the trivial one.}. In this paper, we establish the first results in this direction outside of the interpolation setting: we give a complete classification of simple commutative algebras in the Delannoy category.

\subsection{Statement of results}

Let $\GG=\Aut(\bR,<)$ be the group of order-preserving self-bijections of the real numbers, and fix an algebraically closed field $k$. Let $\uRep(\GG)$ denote the Delannoy category over $k$; see \S \ref{s:delannoy}. The set $\bR^{(n)}$ of tuples $(x_1, \ldots, x_n) \in \bR^n$ with $x_1<\cdots<x_n$ carries a natural transitive action of $\GG$, and the basic objects of the Delannoy category are the Schwartz spaces $\cC(\bR^{(n)})$; see \S \ref{ss:schwartz}. By construction, these spaces are naturally commutative algebras, and as such, are simple and \'etale. There are no other obvious examples of simple algebras, so it is reasonable to suspect that they exhaust the class. This is our main result:

\begin{maintheorem} \label{mainthm}
Any simple commutative algebra in $\uRep(\GG)$ is isomorphic to some $\cC(\bR^{(n)})$.
\end{maintheorem}

We also prove two ``relative'' classification results that generalize Theorem~\ref{mainthm}. Let $\fT$ be an arbitrary pre-Tannakian category. We classify certain types of algebras in the Deligne tensor product $\uRep(\GG) \boxtimes \fT$, in terms of information about $\fT$. The first result deals with \'etale algebras, and requires no assumptions on $\fT$.

\begin{maintheorem} \label{mainthm2}
A simple commutative \'etale algebra in $\uRep(\GG) \boxtimes \fT$ is isomorphic to $\cC(\bR^{(n)}) \boxtimes E$, for some $n$ and some simple commutative \'etale algebra $E$ in $\fT$.
\end{maintheorem}

The second generalization looks at simple algebras, and requires some assumptions.

\begin{maintheorem} \label{mainthm3}
Suppose $\fT$ is semi-simple and all its simple commutative algebras are \'etale. Then any simple commutative algebra in $\uRep(\GG) \boxtimes \fT$ is \'etale, and therefore of the form $\cC(\bR^{(n)}) \boxtimes E$, as above.
\end{maintheorem}

For example, Theorem~\ref{mainthm3} shows that the simple commutative algebras in $\uRep(\GG) \boxtimes \uRep(\GG)$ are of the form $\cC(\bR^{(n)}) \boxtimes \cC(\bR^{(m)})$.

\begin{remark}
Let $\fT$ be a pre-Tannakian category and let $\Et(\fT)$ be the category of \'etale algebras in $\fT$. In \cite{discrete} it is shown that $\Et(\fT)^{\op}$ is a pre-Galois category, and thus, by the main result of \cite{pregalois}, it is equivalent to the category $\bS(G)$ of finitary smooth $G$-sets for some pro-oligomorphic group $G$. In \cite{discrete}, the \defn{oligomorphic fundamental group} of $\fT$ is defined to be this $G$ (see \S \ref{ss:oligo-fund} for details). Theorem~\ref{mainthm} can be rephrased as saying that the oligomorphic fundamental group of $\uRep(\GG)$ is $\GG$.
\end{remark}

\begin{remark}
Our assumption that $k$ is algebraically closed can be relaxed. If $k$ is perfect then the simple commutative algebras in $\uRep(\GG)$ have the form $K \otimes \cC(\bR^{(n)})$, where $K/k$ is a finite field extension. This statement can be deduced from Theorem~\ref{mainthm} using descent theory; note that the algebras $\cC(\bR^{(n)})$ have trivial automorphism groups, and so there are no twisted forms. Theorems~\ref{mainthm2} and~\ref{mainthm3} hold as written for arbitrary $k$; note, however, that the assumption in Theorem~\ref{mainthm3} that simple algebras in $\fT$ are \'etale forces $k$ to be perfect. We assume $k$ is algebraically closed merely to simplify the exposition.
\end{remark}

\subsection{Sketch of proof}

We explain the proof of Theorem~\ref{mainthm}. Let $A$ be a simple commutative algebra in $\uRep(\GG)$. The main step in the proof is to show that there is some open subgroup $U$ of $\GG$ such that the restriction of $A$ to $\uRep(U)$ admits an algebra homomorphism to $\bone$, the trivial algebra. Given this, by adjunction, we obtain an algebra homomorphism $A \to \cC(\GG/U)$ in $\uRep(\GG)$, which is necessarily injective since $A$ is simple. Note that $\GG/U$ is isomorphic to $\bR^{(n)}$ for some $n$. We are then reduced to showing that a simple subalgebra of $\cC(\bR^{(n)})$ is of the form $\cC(\bR^{(m)})$. To do this, we first show that such algebras are \'etale (see \S \ref{ss:et-sub}), and then appeal to a general classification of \'etale subalgebras of \'etale algebras from \cite{discrete}. Ultimately, this shows that $A$ is isomorphic to $\cC(X)$, where $X$ is a quotient $\GG$-set of $\bR^{(n)}$, and such an $X$ must have the form $\bR^{(m)}$ by the classification of $\GG$-sets.

We now say a few words about the main step in the proof, which is carried out in \S \ref{s:proof}. We must produce an open subgroup $U$ and an algebra homomorphism $\Res^{\GG}_U(A) \to \bone$. Instead of doing this in one step, we first find a $U$ and an algebra homomorphism $\Res^{\GG}_U(A) \to B$, where $B$ is an algebra in $\uRep(U)$ that is smaller than $A$, in an appropriate sense. By iterating this procedure, we eventually obtain a homomorphism to $\bone$. This argument is somewhat delicate, and relies on the explicit rule for how simples in the Delannoy category restrict to open subgroups, which was obtained in \cite{line}.

Essentially the same argument is used to prove the first part of Theorem~\ref{mainthm3}; the second part of that theorem follows from Theorem~\ref{mainthm2}. To prove Theorem~\ref{mainthm2}, we establish a special group-theoretic property of $\GG$ (it is \emph{split}, see \S \ref{ss:split}), and then translate the group theory into algebra using the theory of the oligomorphic fundamental group from \cite{discrete}.

\subsection{Tensor category terminology} \label{ss:tencat}

A \defn{tensor category} is an additive $k$-linear category equipped with a symmetric monoidal structure that is $k$-bilinear. We write $\bone$ for the unit of a tensor category (the ``trivial representation'') and $\Gamma$ for the invariants functor, defined by $\Gamma(X)=\Hom(\bone, X)$. A \defn{tensor functor} is a symmetric monoidal $k$-linear functor. We write $X^{\vee}$ for the dual of $X$ and $\udim{X}$ for the categorical dimension of $X$, when $X$ is a rigid object (i.e., it admits a dual). A \defn{pre-Tannakian category} is an abelian tensor category in which all objects have finite length, all $\Hom$ spaces are finite dimensional, all objects are rigid, and $\End(\bone)=k$.

\subsection{Notation}

We list the most important notation:
\begin{description}[align=right,labelwidth=1.9cm,leftmargin=!]
\item[ $k$ ] the coefficient field (algebraically closed in \S \ref{s:proof} and \S \ref{s:proof2})
\item[ $\bone$ ] the tensor unit, or trivial representation
\item[ $\Gamma$ ] the invariants functor $\Hom(\bone, -)$
\item[ $\fT$ ] an arbitrary pre-Tannakian category
\item[ $\GG$ ] the Delannoy group $\Aut(\bR, <)$
\item[ $\uRep(\GG)$ ] the Delannoy category
\end{description}

\section{Algebras in tensor categories} \label{s:alg}

In this section, we establish some general results about algebras in tensor categories. We fix a pre-Tannakian category $\fT$ for the duration of \S \ref{s:alg}. We do not require $k$ to be algebraically closed in \S \ref{s:alg}.

\subsection{Algebras}

By an \defn{algebra} in $\fT$ we will always mean a commutative, unital, and associative algebra. Let $A$ be an algebra in $\fT$. An \defn{ideal} of $A$ is an $A$-submodule of $A$. We say that $A$ is \defn{simple} if it is non-zero and its only ideals are~the zero ideal and the unit ideal. We note that $\Gamma(A)$ is always a finite dimensional $k$-algebra.

\begin{proposition} \label{prop:Gamma-field}
If $A$ is a simple algebra then $\Gamma(A)$ is a finite extension field of $k$.
\end{proposition}

\begin{proof}
Let $x \in \Gamma(A)$ be non-zero. Let $m_x \colon A \to A$ be the multiplication-by-$x$ map. The kernel of $m_x$ is an ideal of $A$, and is not the unit ideal since it does not contain the unit. It is therefore the zero ideal, since $A$ is simple. Thus $m_x$ is injective, and therefore an isomorphism since $A$ has finite length. It follows that the map $\Gamma(m_x) \colon \Gamma(A) \to \Gamma(A)$ is also an isomorphism, and so $x$ is a unit of the ring $\Gamma(A)$, as required.
\end{proof}

\begin{corollary} \label{cor:Gamma-field}
If $\fT$ is semi-simple, $k$ is algebraically closed, and $A$ is a simple algebra in $\fT$ then the trivial representation appears in $A$ with multiplicity one.
\end{corollary}

\begin{proof}
Indeed, $\Gamma(A)=k$ and so the result follows.
\end{proof}

\begin{remark}
Semi-simplicity is necessary in Corollary~\ref{cor:Gamma-field}. Indeed, suppose $k$ has positive characteristic $p$, $\fT=\Rep(\fS_p)$, and $A$ is the algebra of functions on the $\fS_p$-set $\{1, \ldots, p\}$ (under pointwise multiplication). Then $A$ is a simple algebra (it is \'etale and $\Gamma(A)=k$, see \S \ref{ss:etale} below), but the trivial representation appears in it with multiplicity two (once as a subobject and once as a quotient). There is a similar example in characteristic~0: if $\fT$ is the abelian envelope of $\uRep(\fS_t)$ with $t=0$ then the defining object $A$ is a simple \'etale algebra in which the trivial representation appears with multiplicity two.
\end{remark}

\subsection{Frobenius algebras}

Let $A$ be an algebra in $\fT$ and suppose that $\lambda \colon A \to \bone$ is a map in $\fT$. We then obtain a symmetric bilinear form
\begin{displaymath}
(,)_{\lambda} \colon A \otimes A \to \bone, \qquad
(x,y)_{\lambda}=\lambda(xy).
\end{displaymath}
We say that $(A, \lambda)$ is a \defn{Frobenius algebra} if this pairing is perfect, i.e., it identifies $A$ with its dual.

\begin{proposition} \label{prop:simple-Frob}
Suppose that $A$ is a simple algebra and $\lambda \colon A \to \bone$ is a non-zero map. Then $(,)_{\lambda}$ is perfect, and so $(A, \lambda)$ is a Frobenius algebra.
\end{proposition}

\begin{proof}
Let $\fn \subset A$ be the kernel of the form $(,)_{\lambda}$, i.e., the maximal subobject of $A$ such that $(,)_{\lambda}$ vanishes identically on $\fn \otimes A$. One easily sees that $\fn$ is an ideal of $A$. It is not the unit ideal since the restriction of $(,)_{\lambda}$ to $\bone \otimes A=A$ is $\lambda$, which is non-zero. Since $A$ is simple, it follows that $\fn=0$, and so $(,)_{\lambda}$ is perfect, as required.
\end{proof}

\begin{remark}
Suppose we are in the context of Proposition~\ref{prop:simple-Frob} and $k$ is algebraically closed. Then $\Gamma(A)=k$ by Proposition~\ref{prop:Gamma-field}, and so $\Hom_{\fT}(A, \bone)=k$ since $A$ is self-dual. It follows that the Frobenius algebra structure on $A$ is unique up to scaling.
\end{remark}

\begin{corollary} \label{cor:simple-Frob}
Suppose that $\fT$ is semi-simple and $A$ is a simple algebra in $\fT$. Then $A$ admits the structure of a Frobenius algebra.
\end{corollary}

\begin{proof}
The unit morphism $\bone \to A$ splits since $\fT$ is semi-simple, and so there exists some non-zero map $\lambda \colon A \to \bone$.
\end{proof}

\subsection{\'Etale algebras} \label{ss:etale}

Let $A$ be an algebra in $\fT$. Let $\ul{\End}(A)=A^{\vee} \otimes A$ be the internal endomorphism algebra of $A$. There is a natural map $A \to \ul{\End}(A)$, coming from the action of $A$ on itself. Composing this with the trace map $\ul{\End}(A)\to \bone$, we obtain a map
\begin{displaymath}
\epsilon_A \colon A \to \bone
\end{displaymath}
called the \defn{trace map} for $A$. The \defn{trace form} on $A$ is the symmetric bilinear form $(,)_{\epsilon_A}$. We say that $A$ is \defn{\'etale} if the trace form is a perfect pairing. In other words, $A$ is \'etale if it is a Frobenius algebra with respect to the trace map. In what follows, we recall some basic facts about \'etale algebras.

\textit{(a) Permanence properties.} The class of \'etale algebras is closed under direct products and tensor products \cite[\S 4.2]{discrete}. Any quotient of an \'etale algebra is again \'etale \cite[Corollary~5.2]{discrete}. In particular, a direct factor of an \'etale algebra is again \'etale.

\textit{(b) Invariant algebras.} If $A$ is an \'etale algebra in $\fT$ then $\Gamma(A)$ is a finite \'etale $k$-algebra \cite[Corollary~5.3]{discrete}; explicitly, this means that $\Gamma(A)$ is a finite product of finite separable extensions of $k$. Moreover, ideals of $\Gamma(A)$ correspond bijectively to ideals of $A$, via $I \mapsto IA$ \cite[Corollary~5.2]{discrete}. This has two notable consequences. First, an \'etale algebra $A$ is simple if and only if $\Gamma(A)$ is a field. And second, an \'etale algebra $A$ decomposes into a product of simple \'etale algebras; the simple factors of $A$ correspond to the primitive idempotents of $\Gamma(A)$. This decomposition is unique (up to permuting factors).

\textit{(c) Simple algebras.} Let $A$ be a simple algebra. Then $A$ is \'etale if and only if the trace map $\epsilon_A$ is not identically zero. Indeed, if $A$ is \'etale then clearly $\epsilon_A$ is non-zero; conversely, if $\epsilon_A$ is non-zero then $A$ is \'etale by Proposition~\ref{prop:simple-Frob}. In particular, if $\udim(A)$ is non-zero then $A$ is \'etale, as $\udim(A)=\epsilon_A(1)$.

The final statement has a partial converse: if $\fT$ is semi-simple, $k$ is algebraically closed, and $A$ is a simple \'etale algebra then $\udim(A)$ is non-zero. Indeed, the trivial representation appears in $A$ with multiplicity one (Corollary~\ref{cor:Gamma-field}), and so the trace map $\epsilon_A$ must be the left inverse of the unit, up to a non-zero scalar. Thus $\udim(A) = \epsilon_A(1) \ne 0$.

\textit{(d) External products.} Let $\fT'$ be a second pre-Tannakian category, and let $A$ be an algebra of $\fT$. Then $B=A \boxtimes \bone$ is an algebra of $\fT \boxtimes \fT'$. One easily sees that $\epsilon_B=\epsilon_A \boxtimes \id$, from which it follows that $A$ is \'etale if and only if $B$ is \'etale.

Let $A$ be an \'etale algebra of $\fT$, and let $A'$ be an \'etale algebra of $\fT'$. Then $A \boxtimes A'$ is \'etale, as it is the tensor product of the \'etale algebras $A \boxtimes \bone$ and $\bone \boxtimes A'$. We have
\begin{displaymath}
\Gamma(A \boxtimes A') = \Gamma(A) \otimes_k \Gamma(A')
\end{displaymath}
by \cite[Prop.~1.11.2]{EGNO}. In particular, if $A$ is absolutely simple, i.e., $\Gamma(A)=k$, and $A'$ is simple then $\Gamma(A \boxtimes A')$ is a field, and so $A \boxtimes A'$ is simple.

\subsection{Sub\'etale algebras} \label{ss:et-sub}

We say that an algebra in a pre-Tannakian category $\fT$ is \defn{sub\'etale} if it is isomorphic to a subalgebra of an \'etale algebra. In $\Vec$, a sub\'etale algebra is \'etale: indeed, an \'etale algebra is a product of finite separable extensions of $k$, and a subextension of a separable extension is separable. It follows that in classical representation categories, and their interpolation categories (such as Deligne's category), any sub\'etale algebra is \'etale\footnote{To see this in interpolation categories, it is easiest to use the ultraproduct approach \cite{Harman, interp}. For instance, $\uRep(\fS_t)$ can be realized as a subcategory of the ultraproduct of the categories $\Rep(\fS_n)$ as $n$ (and the coefficient field) varies. Suppose $A=(A_n)$ is an \'etale algebra in $\uRep(\fS_t)$, where each $A_n$ is an object of $\Rep(\fS_n)$, and suppose $B=(B_n)$ is a subalgebra of $A$. Then for all $n$ belonging to some set in the ultrafilter, $A_n$ is an \'etale algebra and $B_n$ is a subalgebra, and therefore \'etale; thus $B$ is \'etale.}. We will see, however, that the Delannoy category contains sub\'etale algebras that are not \'etale (Example~\ref{ex:subet}). We now give some criteria for showing that a sub\'etale algebra is in fact \'etale.

Let $B$ be an \'etale algebra of $\fT$, and let $A$ be a subalgebra of $B$. Consider the sequence
\begin{displaymath}
\xymatrix{
0 \ar[r] & A \ar[r]^-i & B \ar[r]^-j & B \otimes_A B, }
\end{displaymath}
where $i$ is the inclusion and $j(x)=x \otimes 1 - 1 \otimes x$. Clearly, $j \circ i=0$.

\begin{proposition} \label{prop:et-sub}
$A$ is \'etale if and only if this sequence is exact, i.e., $\im(i)=\ker(j)$.
\end{proposition}

\begin{proof}
This is essentially proved in \cite[\S 5.3]{discrete}, as we now explain. First, if $A$ is \'etale then the sequence is shown to be exact in the proof of \cite[Proposition~5.12]{discrete}. We now prove the converse. Let $I$ be the kernel of $B \otimes B \to B \otimes_A B$. Let $\gamma$ be the unique idempotent of $\Gamma(B \otimes B)$ such that $\ker(\gamma)=I$, i.e., $1-\gamma$ generates $I$; this is an ``E-idempotent'' by \cite[Lemma~5.14]{discrete}. The kernel of $j$ is the subobject $B^{\gamma}$ of $B$ defined in \cite[\S 5.3]{discrete}. By \cite[Proposition~5.12]{discrete}, $B^{\gamma}$ is an \'etale subalgebra of $B$. If the sequence is exact, then $A=B^{\gamma}$ is \'etale.
\end{proof}

In the above proof, if we work in the category of vector spaces, then the vanishing locus $V(I) \subset \Spec(B) \times \Spec(B)$ is an equivalence relation on $\Spec(B)$, and the quotient space is $\Spec(B^{\gamma})$. Essentially the arguments in \cite{discrete} show that similar reasoning applies in general pre-Tannakian categories. We will revisit these ideas in more detail in \S \ref{ss:schwartz}.

\begin{proposition} \label{prop:etale-summand}
If $A$ is a summand of $B$ as an $A$-module then $A$ is \'etale.
\end{proposition}

\begin{proof}
Write $B=A \oplus M$, where $M$ is an $A$-submodule of $B$. We have
\begin{displaymath}
B \otimes_A B = A \oplus M \oplus M \oplus (M \otimes_A M).
\end{displaymath}
In terms of this decomposition, we have $j \vert_M = (0, \id_M, -\id_M, 0)$. Thus $\ker(j \vert_M)=0$, and so $\ker(j)=\im(i)$, as required.
\end{proof}

\begin{proposition} \label{prop:simple-sub}
If $A$ is simple and there is a map $\lambda \colon B \to \bone$ in $\fT$ that has non-zero restriction to $A$ then $A$ is \'etale.
\end{proposition}

\begin{proof}
Let $B=\bigoplus_{i=1}^n B_i$ be the decomposition of $B$ into simple \'etale algebras, and let $\lambda_i$ be the restriction of $\lambda$ to $B_i$. The restriction of $\lambda_i$ along the projection map $A \to B_i$ is non-zero for some $i$. Moreover, the map $A \to B_i$ is injective since $A$ is simple. We may thus replace $B$ with this $B_i$ and thereby assume that $B$ is simple as well.

The pairing $(,)_{\lambda}$ on $B$ is perfect, and restricts to a perfect pairing on $A$; both claims follow from Proposition~\ref{prop:simple-Frob}. Thus $B=A \oplus A^{\perp}$, where $A^{\perp}$ is the orthogonal complement of $A$ under $(,)_{\lambda}$. Since this decomposition is one of $A$-modules, the result follows from Proposition~\ref{prop:etale-summand}.
\end{proof}

\begin{corollary} \label{cor:simple-sub}
If $A$ is simple and $\udim(B) \ne 0$ then $A$ is \'etale.
\end{corollary}

\begin{proof}
Take $\lambda=\epsilon_B$. Since $\lambda(1)=\udim(B)$ is non-zero, $\lambda$ has non-zero restriction to $A$.
\end{proof}

\begin{corollary} \label{cor:simple-sub}
If $\fT$ is semi-simple and $A$ is simple then $A$ is \'etale.
\end{corollary}

\begin{proof}
Indeed, the unit $\bone \to A$ admits a right inverse $\lambda_0 \colon A \to \bone$, which extends to a map $\lambda \colon B \to \bone$; both claims follow from semi-simplicity.
\end{proof}

\begin{remark} \label{rmk:simple-subet}
We can rephrase the above corollary as follows: in a semi-simple pre-Tannakian category, any simple sub\'etale algebra is \'etale.
\end{remark}

If $\fT_1$ and $\fT_2$ are pre-Tannakian categories and $Y$ and $Z$ are objects of $\fT_1 \boxtimes \fT_2$, then we define the \defn{partial internal Hom}, denoted $\uHom_{\fT_1}(Y, Z)$ to be the object of $\fT_2$ characterized by the mapping property
\begin{displaymath}
\Hom_{\fT_2}(X, \uHom_{\fT_1}(Y, Z)) = \Hom_{\fT_1 \boxtimes \fT_2}(Y \otimes X, Z),
\end{displaymath}
where $X$ is an object of $\fT_2$. If $Z$ is a commutative algebra in $\fT_1 \boxtimes \fT_2$, then one easily sees that $\uHom_{\fT_1}(\bone, Z)$ is naturally a commutative algebra in $\fT_2$.

\begin{proposition} \label{prop:et-invar}
Let $\fT=\fT_1 \boxtimes \fT_2$, where $\fT_1$ and $\fT_2$ are pre-Tannakian and $\fT_1$ is semi-simple, and let $B$ be an \'etale algebra of $\fT$. Then $\uHom_{\fT_1}(\bone, B)$ is an \'etale algebra in $\fT_2$.
\end{proposition}

\begin{proof}
Let $\{L_i\}_{i \in I}$ be the simple objects of $\fT_1$, and let $0 \in I$ be the index such that $L_0=\bone$. Write $B=\bigoplus_{i \in I} L_i \boxtimes B_i$. Then $A=\bone \boxtimes B_0$ is a subalgebra of $B$, and each $L_i \boxtimes B_i$ is an $A$-submodule of $B$. It follows that $A$ is a summand of $B$ as an $A$-module, and so $A$ is \'etale by Proposition~\ref{prop:etale-summand}. It follows that $B_0$ is \'etale (\S \ref{ss:etale}(d)), which completes the proof.
\end{proof}

\section{The Delannoy category} \label{s:delannoy}

In this section we briefly recall the general theory of oligomorphic tensor categories from \cite{repst}, and then recall details about the Delannoy category from \cite{line}. We also prove a few simple new results. We do not require $k$ to be algebraically closed in \S \ref{s:delannoy}.

\subsection{Oligomorphic tensor categories} \label{ss:oligo}

A \defn{pro-oligomorphic group} is a topological group $G$ that is (a) non-archimedean (open subgroups form a neighborhood basis of the identity); (b) Roelcke pre-compact (if $U$ and $V$ are open subgroups then $U \backslash G/V$ is finite); and (c) Hausdorff. An \defn{oligomorphic permutation group} is a permutation group $(G, \Omega)$ such that $G$ has finitely many orbits on $\Omega^n$ for all $n$.

These two notions are closely related. Suppose $(G, \Omega)$ is an oligomorphic permutation group. For a finite subset $A \subset \Omega$, let $G(A)$ be the subgroup of $G$ fixing each element of $A$. Then the $G(A)$'s form a neighborhood basis of the identity for a pro-oligomorphic topology on $G$ \cite[Proposition~2.4]{repst}. Conversely, if $G$ is pro-oligomorphic then for every open subgroup $U$, the action of $G$ on $G/U$ is oligomorphic (meaning there are finitely many orbits on every cartesian power; see below), and this realizes $G$ as a dense subgroup of an inverse limit of oligomorphic permutation groups.

Let $G$ be a pro-oligomorphic group. We say that an action of $G$ on a set $X$ is \defn{smooth} if the stabilizer of any element of $X$ is an open subgroup of $G$, and \defn{finitary} if $G$ has finitely many orbits on $X$. We use the term ``$G$-set'' to mean ``set equipped with a finitary smooth action of $G$.'' We let $\bS(G)$ be the category of (finitary smooth) $G$-sets. It is closed under products \cite[Proposition~2.8]{repst}, from which it easily follows that $\bS(G)$ has all finite limits and finite co-limits (and they are computed on the underlying set). Additionally, we see that the action of $G$ on any $G$-set is oligomorphic.

A \defn{measure} for $G$ is a rule $\mu$ assigning to each morphism $f \colon Y \to X$ in $\bS(G)$, with $X$ transitive, a quantity $\mu(f)$ in $k$ such that some axioms hold; see \cite[\S 3.6]{repst}. Given a measure $\mu$, we define a tensor category $\uPerm(G, \mu)$ as follows. The objects are formal symbols\footnote{In \cite{repst}, we give a concrete interpretation of this category in which $\cC(X)$ is the vector space whose elements are ``Schwartz functions'' on $X$.} $\cC(X)$ where $X$ is a $G$-set; we call $\cC(X)$ the \defn{Schwartz space} on $X$. A morphism $\cC(X) \to \cC(Y)$ is a $G$-invariant $k$-valued function on $Y \times X$. Composition is given by convolution of functions, the definition of which involves $\mu$. We have
\begin{displaymath}
\cC(X) \oplus \cC(Y) = \cC(X \amalg Y), \qquad
\cC(X) \otimes \cC(Y) = \cC(X \times Y).
\end{displaymath}
Intuitively, $\cC(X)$ behaves like a permutation representation with basis $X$. The object $\cC(X)$ is rigid and self-dual, and its categorical dimension is $\mu(X)$, which is shorthand for $\mu(X \to \mathrm{pt})$. Moreover, $\cC(X)$ is naturally a commutative algebra, and as such is \'etale\footnote{We note that the definition of \'etale algebra makes sense in any tensor category.}; this follows from \cite[Proposition~9.6]{repst} and \cite[Proposition~4.11]{discrete}. We note that $\uPerm(G, \mu)$ is abelian if and only if $G$ is the trivial group. In particular, if $G$ is non-trivial then $\uPerm(G, \mu)$ is not a pre-Tannakian category. We refer to \cite[\S 8]{repst} for additional details about the category $\uPerm(G, \mu)$.

\begin{example}
Let $\fS$ be the group of all permutations of $\Omega=\{1,2,\ldots\}$. Then $(\fS, \Omega)$ is an oligomorphic permutation group. Suppose $k$ has characteristic~0. For each $t \in k$ there is a unique measure $\mu_t$ such that $\mu_t(\Omega)=t$, and these account for all the $k$-valued measures. The Karoubi envelope of $\uPerm(\fS, \mu_t)$ is equivalent to Deligne's interpolation category $\uRep(\fS_t)$ \cite{Deligne}. See \cite[\S 14]{repst} for proofs of these statements and additional details.
\end{example}

\subsection{The Delannoy group} \label{ss:delgp}

Let $\GG=\Aut(\bR,<)$ be the group of all order-preserving self-bijections of the real numbers. One easily sees that the action of $\GG$ on $\bR$ is oligomorphic. We now explain the structure of this group and the category $\bS(\GG)$ of $\GG$-sets.

Let $\bR^{(n)}$ be the subset of $\bR^n$ consisting of tuples $(x_1, \ldots, x_n)$ with $x_1<\cdots<x_n$. One easily sees that this is a transitive $\GG$-set. In fact, these account for all transitive $\GG$-sets, up to isomorphism \cite[Corollary~16.2]{repst}. One can think of $\bR^{(n)}$ as the space of all order-preserving injections $[n] \to \bR$, where $[n]=\{1, \ldots, n\}$. Thus if $i \colon [m] \to [n]$ is an order-preserving injection then there is an induced map $i^* \colon \bR^{(n)} \to \bR^{(m)}$ of $\GG$-sets. It is not difficult to see that these are the only $\GG$-maps between these two $\GG$-sets. (The key point is that the image $y$ of $x=(1, \ldots, n)$ must be fixed by the stabilizer of $x$, which implies that the coordinates of $y$ belong to $\{1, \ldots, n\}$.) Thus the category of transitive objects in $\bS(\GG)$ is anti-equivalent to the category of finite totally ordered sets, with order-preserving injections maps.

Recall that if $A$ is a finite subset of $\bR$ then $\GG(A)$ is the subgroup of $\GG$ consisting of elements fixing each point of $A$. By definition, a subgroup of $\GG$ is open if it contains some $\GG(A)$. In fact, every open subgroup of $\GG$ has the form $\GG(A)$ \cite[Proposition~16.1]{repst}; this is essentially equivalent to the classification of transitive sets discussed above. 

Let $A$ be an $n$-element subset of $\bR$. Then $\GG/\GG(A)$ is isomorphic to $\bR^{(n)}$. In particular, if $A'$ is another $n$-element subset of $\bR$ then $\GG(A)$ and $\GG(A')$ are conjugate subgroups of $\GG$. The set $\bR \setminus A$ consists of $n+1$ intervals, each of which is isomorphic to $\bR$ as a totally ordered set. From this, one sees that $\GG(A)$ is isomorphic to $\GG^{n+1}$. We write $\GG(0)$ in place of $\GG(\{0\})$; this group will occur frequently in what follows.

Let $A$ and $B$ be finite subsets of $\bR$. One easily sees that $\GG(B) \subset \GG(A)$ if and only if $A \subset B$. Suppose this is the case. We can then find a chain $A=B_0 \subset \cdots \subset B_r=B$ such that $B_{i+1}$ is obtained from $B_i$ by adding a single point. If $B_i$ has $m$ elements then $\GG(B_i)$ is isomorphic to $\GG^{m+1}$. Under this isomorphism, the subgroup $\GG(B_{i+1})$ is identified with $\GG^i \times \GG(0) \times \GG^{m-i}$ for some $i$ (determined by which interval of $\bR \setminus B_i$ the new point of $B_{i+1}$ belongs to). Thus, for many purposes, understanding containments between general open subgroups of $\GG$ reduces to the case of the containment $\GG(0) \subset \GG$.

\subsection{The Delannoy category} \label{ss:delcat}

The group $\GG$ has a unique $k$-valued measure $\mu$ satisfying $\mu(\bR^{(n)})=(-1)^n$ \cite[Theorem~16.7]{repst}; in total, $\GG$ has four $k$-valued measures, but the other three will not be relevant to us. We define the \defn{Delannoy category} $\uRep(\GG)$ to be the Karoubi envelope of $\uPerm(\GG, \mu)$. It is a semi-simple pre-Tannakian category \cite[\S 16.5]{repst}. It is remarkable that this category is semi-simple over any coefficient field.

More generally, for any open subgroup $U$ of $\GG$, we define $\uRep(U)$ to be the Karoubi envelope of $\uPerm(U, \mu)$, which is again a semi-simple pre-Tannakian category. If $U$ is isomorphic to $\GG^s$ then $\uRep(U)$ is equivalent to the Deligne tensor power $\uRep(\GG)^{\boxtimes s}$ of the Delannoy category. See \cite[\S 4.6]{EGNO} for background on this construction.

\subsection{Schwartz spaces} \label{ss:schwartz}

Let $X$ be a $\GG$-set. The Schwartz space $\cC(X)$ is then an \'etale algebra in the category $\uRep(\GG)$; indeed, in \cite[\S 9.6]{repst} it is shown that $\cC(X)$ has the structure of a special commutative Frobenius algebra, and such algebras are equivalent to \'etale algebras \cite[Proposition~4.11]{discrete}. The invariant algebra $\Gamma(\cC(X))$ consists of $\GG$-invariant $k$-valued functions on $X$. The algebra structure on this is given by pointwise multiplication of functions; this follows from the definition of the algebra structure on $\cC(X)$ given in \cite[Proposition~9.6]{repst}. In particular, $\cC(X)$ is a simple algebra if and only if $X$ is transitive. We thus have simple \'etale alebras $\cC(\bR^{(n)})$ for each $n \ge 0$.

Theorem~\ref{mainthm} asserts that these are the only simple algebras in $\uRep(\GG)$, up to isomorphism. The more substantial part of the proof (carried out in \S \ref{s:proof}) is showing that any simple algebra is a subalgebra of some $\cC(\bR^{(n)})$. Given this, we are reduced to showing that any simple subalgebra of $\cC(\bR^{(n)})$ is isomorphic to some $\cC(\bR^{(m)})$. We take care of this part now.

In \cite[\S 5.3]{discrete}, we classify the \'etale subalgebras of an \'etale algebra in an arbitrary pre-Tannakian category in terms of elements we call E-idempotents (these appeared in the proof of Proposition~\ref{prop:et-sub} above). We recall the definition. Let $B$ be an \'etale algebra in some pre-Tannakian category $\fT$. An \defn{E-idempotent} of $B$ is an idempotent $\gamma \in \Gamma(B \otimes B)$ satisfying the following conditions:
\begin{enumerate}
\item We have $\mu(\gamma)=1$, where $\mu \colon B \otimes B \to B$ is the multiplication map.
\item We have $\tau(\gamma)=\gamma$, where $\tau$ is the symmetry of $B \otimes B$.
\item For $1 \le i<j \le 3$, let $\gamma_{i,j}$ be the idempotent of $B \otimes B \otimes B$ that puts $\gamma$ in the $i$ and $j$ factors, and 1 in the remaining factor. Then $\gamma_{1,2} \gamma_{2,3}=\gamma_{1,2} \gamma_{1,3} = \gamma_{1,3} \gamma_{2,3}$.
\end{enumerate}
By \cite[Proposition~5.12]{discrete}, \'etale subalgebras of $B$ are in natural bijective correspondence with E-idempotents of $B$. The following proposition classifies E-idempotents on Schwartz spaces, and, at the same time, provides intuition for the definition.

\begin{proposition} \label{prop:E-idemp}
Let $X$ be a $\GG$-set. Then E-idempotents of $\cC(X)$ correspond bijectively to $\GG$-stable equivalence relations on $X$.
\end{proposition}

\begin{proof}
Let $\gamma$ be an E-idempotent of $\cC(X) \otimes \cC(X)=\cC(X \times X)$. Thus $\gamma$ is a $\GG$-invariant function on $X \times X$ that takes values in $\{0, 1\}$ (since it is idempotent). It is therefore the indicator function of a $\GG$-stable subset $R \subset X \times X$. Condition (a) implies that $R$ contains the diagonal, condition (b) implies that $R$ is symmetric, and condition (c) implies that $R$ is transitive. We elaborate on transitivity. The element $\gamma_{i,j}$ is the pull-back of $\gamma$ under the projection map $p_{i,j} \colon X^3 \to X^2$ on the $i$ and $j$ coordinates. It is the indicator function of $p_{i,j}^{-1}(R)$. Thus condition (c) translates to
\begin{displaymath}
p_{1,2}^{-1}(R) \cap p_{2,3}^{-1}(R) = p_{1,2}^{-1}(R) \cap p_{1,3}^{-1}(R) = p_{1,3}^{-1}(R) \cap p_{2,3}^{-1}(R),
\end{displaymath}
which is equivalent to transitivity. We thus see that an E-idempotent is the indicator function of an equivalence relation. This reasoning is reversible, which completes the proof.
\end{proof}

With the above proposition in hand, we can now classify the simple subalgebras of $\cC(\bR^{(n)})$.

\begin{proposition} \label{prop:schwartz-sub}
A simple subalgebra of $\cC(\bR^{(n)})$ is isomorphic to $\cC(\bR^{(m)})$ for some $m \le n$.
\end{proposition}

\begin{proof}
Let $\gamma$ be an E-idempotent of $\cC(\bR^{(n)})$. By Proposition~\ref{prop:E-idemp}, $\gamma$ corresponds to a $\GG$-stable equivalence relation $R \subset \bR^{(n)} \times \bR^{(n)}$. Let $X=\bR^{(n)}/R$ be the quotient set, and let $p \colon \bR^{(n)} \to X$ be the quotient map. Then $p^* \colon \cC(X) \to \cC(\bR^{(n)})$ is an injective algebra homomorphism, and so we can identify $\cC(X)$ with a subalgebra of $\cC(\bR^{(n)})$. It is not difficult to see that this subalgebra corresponds to $\gamma$ under the bijection of \cite[Proposition~5.12]{discrete}, and so the $\cC(X)$ account for all \'etale subalgebras of $\cC(\bR^{(n)})$. By the classification of transitive $\GG$-sets (\S \ref{ss:delgp}), we have $X \cong \bR^{(m)}$ for some $m \le n$. Finally, any simple subalgebra of $\cC(\bR^{(n)})$ is \'etale by Corollary~\ref{cor:simple-sub}.
\end{proof}

\subsection{Simple objects} \label{ss:simple}

A \defn{weight} is a word in the two letter alphabet $\{\wa, \wb\}$. The simple objects of $\uRep(\GG)$ are indexed by weights \cite[\S 4.1]{line}. For a weight $\lambda$, we write $L_{\lambda}$ for the corresponding simple object; this is defined as a certain summand of $\cC(\bR^{(n)})$, where $n=\ell(\lambda)$ is the length of $\lambda$. If $\lambda$ is the empty word then $L_{\lambda}$ is the trivial representation $\bone$. Write $\lambda^{\vee}$ for the word obtained from $\lambda$ by changing all $\wa$ to $\wb$, and vice versa. The dual of $L_{\lambda}$ is $L_{\lambda^{\vee}}$ \cite[Proposition~4.16]{line}. The dimension of $L_{\lambda}$ is $(-1)^{\ell(\lambda)}$ \cite[Corollary~5.7]{line}.

The simple decomposition of Schwartz space is given in \cite[Theorem~4.7]{line}:
\begin{displaymath}
\cC(\bR^{(n)}) = \bigoplus_{\ell(\lambda) \le n} L_{\lambda}^{\oplus m(\lambda)}, \qquad m(\lambda) = \binom{n}{\ell(\lambda)},
\end{displaymath}
In particular, for $n=1$, we find
\begin{displaymath}
\cC(\bR) = L_{\wa} \oplus \bone \oplus L_{\wb}.
\end{displaymath}

In \cite[\S 7]{line}, the decomposition of the tensor product $L_{\lambda} \otimes L_{\mu}$ is described for arbitrary weights $\lambda$ and $\mu$. Only a few simple cases are relevant to us:
\begin{align*}
L_{\wa} \otimes L_{\wa} &= L_{\wa\wa}^{\oplus 2} \oplus L_{\wa} \\
L_{\wb} \otimes L_{\wb} &= L_{\wb\wb}^{\oplus 2} \oplus L_{\wb} \\
L_{\wa} \otimes L_{\wb} &= L_{\wa\wb} \oplus L_{\wb\wa} \oplus L_{\wa} \oplus L_{\wb} \oplus \bone.
\end{align*}
The following important example is an immediate consequence of this.

\begin{example} \label{ex:subet}
Since $L_{\wa} \otimes L_{\wa}$ does not contain $L_{\wb}$, it follows that $A=L_{\wa} \oplus \bone$ is a subalgebra of $\cC(\bR)$. Since $A$ is not self-dual as an object of $\uRep(\GG)$, it cannot be an \'etale algebra. Thus $A$ is a sub\'etale algebra that is not \'etale. This example resolves \cite[Question~5.15]{discrete}, which asked if sub\'etale implies \'etale. The existence of this algebra, combined with the discussion at the beginning of \S \ref{ss:et-sub}, implies that $\uRep(\GG)$ cannot be obtained by interpolating Tannakian categories; this was proved in \cite[\S 16]{repst} by other means. It is not difficult to give many other examples of sub\'etale algebras in $\uRep(\GG)$ that are not \'etale, using the same idea.
\end{example}

\subsection{Restriction} \label{ss:res}

For open subgroups $V \subset U$ of $\GG$, we have a restriction functor
\begin{displaymath}
\Res^U_V \colon \uRep(U) \to \uRep(V).
\end{displaymath}
The most important case for us is $U=\GG$ and $V=\GG(0)$. Restriction is then a functor
\begin{displaymath}
\uRep(\GG) \to \uRep(\GG(0)) \cong \uRep(\GG) \boxtimes \uRep(\GG).
\end{displaymath}
We now describe the explicit rule for decomposing the restriction of a simple object. Let $\lambda=\lambda_1 \cdots \lambda_n$ be a weight of length $n$. We let $\lambda[i,j]=\lambda_i \cdots \lambda_j$ and $\lambda[i,j) = \lambda_i \cdots \lambda_{j-1}$, and similarly define $\lambda(i,j]$; here we use the convention that $\lambda[i,j]=\emptyset$ if $j<i$ and $\lambda[i,j)=\emptyset$ if $j \le i$. We have
\begin{displaymath}
\Res^{\GG}_{\GG(0)}(L_{\lambda}) = \bigoplus_{i=0}^n \big( L_{\lambda[1,i]} \boxtimes L_{\lambda(i,n]} \big) \oplus \bigoplus_{i=1}^n \big( L_{\lambda[1,i)} \boxtimes L_{\lambda(i,n]} \big).
\end{displaymath}
In the first sum, we break $\lambda$ in two by cutting it between letters; in the second sum, we break $\lambda$ in two by deleting letters. The above formula is \cite[Theorem~6.9]{line}. Since every containment between open subgroups can be decomposed into a chain where consecutive inclusions look like $\GG(0) \subset \GG$ (see \S \ref{ss:delgp}), one can determine restrictions between any open subgroups from this rule.

\begin{proposition} \label{prop:G0-invar}
The invariant space $L_{\lambda}^{\GG(0)}$ is one-dimensional if $\ell(\lambda) \le 1$, and vanishes otherwise.
\end{proposition}

\begin{proof}
This follows from the restriction rule given above: $\bone \boxtimes \bone$ appears in the restriction of $L_{\lambda}$ to $\GG(0)$ with multiplicity one if $\ell(\lambda) \le 1$, and with multiplicity zero otherwise.
\end{proof}

\begin{example} \label{ex:res}
Here is a simple example (we omit the $\boxtimes$ symbols):
\begin{displaymath}
\Res^{\GG}_{\GG(0)}(L_{\wa\wb}) = L_{\wa\wb}L_{\emptyset} \oplus L_{\wa}L_{\wb} \oplus L_{\emptyset}L_{\wa\wb} \oplus L_{\wa}L_{\emptyset} \oplus L_{\emptyset}L_{\wb}
\end{displaymath}
\end{example}

\subsection{Algebra maps and adjunction}

Let $U$ be an open subgroup of $\GG$. There is then an induction functor
\begin{displaymath}
\Ind_U^{\GG} \colon \uRep(U) \to \uRep(\GG)
\end{displaymath}
that is both left and right adjoint to the restriction functor \cite[\S 2.5]{line}. This functor satisfies $\Ind_U^{\GG}(\bone)=\cC(\GG/U)$.

\begin{proposition} \label{prop:alg-adjunct-0}
Let $A$ be an algebra in $\uRep(\GG)$ and let $U$ be an open subgroup of $\GG$. Suppose we have an algebra homomorphism $f \colon \Res^{\GG}_U(A) \to \bone$ in $\uRep(U)$. Then there is an algebra homomorphism $g \colon A \to \cC(\GG/U)$ in $\uRep(\GG)$, uniquely characterized by the identity $f=\epsilon \circ g$, where $\epsilon \colon \cC(\GG/U) \to \bone$ evaluates a function at $1 \in \GG/U$.
\end{proposition}

\begin{proof}
By adjunction, a unique map $g$ in $\uRep(\GG)$ exists satisfying $f=\epsilon \circ g$. We must show that $g$ is an algebra homomorphism. Compatibility with the unit follows since $\epsilon \circ g$ is compatible with the unit and $\epsilon \colon \Gamma(\cC(\GG/U)) \to \Gamma(\bone)$ is an isomorphism. We now check compatibility with multiplication. Let $g_1, g_2 \colon A \otimes A \to \cC(\GG/U)$ be the maps $g_1=g \circ \mu$ and $g_2 = \mu \circ (g \otimes g)$, where $\mu$ is the multiplication map in $A$ or $\cC(\GG/U)$. We must show $g_1=g_2$. It suffices, by adjunction, to show that the morphisms $\epsilon \circ g_1$ and $\epsilon \circ g_2$ in $\uRep(U)$ are equal. Consider the diagram in $\uRep(U)$
\begin{displaymath}
\xymatrix@C=3em{
A \otimes A \ar[r]^-{g \otimes g} \ar[d]^{\mu} & \cC(\GG/U) \otimes \cC(\GG/U) \ar[d]^{\mu} \ar[r]^-{\epsilon \otimes \epsilon} \ar[r] & \bone \ar@{=}[d] \\
A \ar[r]^g & \cC(\GG/U) \ar[r]^{\epsilon} & \bone }
\end{displaymath}
The right square commutes since $\epsilon$ is an algebra homomorphism, and the outside square commutes since $f=\epsilon \circ g$ is an algebra homomorphism. We thus see that $\epsilon \circ g_1=\epsilon \circ g_2$, as required.
\end{proof}

The above proposition has a geometric interpretation. The map $f$ corresponds to a $U$-fixed point of $\Spec(A)$. The action of $\GG$ on this point defines a map $\GG/U \to \Spec(A)$, which corresponds to the algebra map $g$. Instead of making these geometric ideas precise, we prefer to simply give the algebraic proof above. The following more general result is proved in the same manner.

\begin{proposition} \label{pro:alg-adjunct}
Let $\fT$ be a pre-Tannakian category, let $A$ be an algebra in $\uRep(\GG) \boxtimes \fT$, let $E$ be an algebra in $\fT$, and let $U$ be an open subgroup of $\GG$. Given an algebra homomorphism $A \to \bone \boxtimes E$ in $\uRep(U) \boxtimes \fT$, there is an algebra homomorphism $A \to \cC(\GG/U) \boxtimes E$ in $\uRep(\GG) \boxtimes \fT$, characterized as in Proposition~\ref{prop:alg-adjunct-0}.
\end{proposition}

\subsection{Simple algebras are \'etale} \label{ss:simple-is-etale}

Assume that $k$ is algebraically closed. The main result of this paper, Theorem~\ref{mainthm}, classifies the simple algebras in $\uRep(\GG)$. A consequence of this classification is that all simple algebras are \'etale. In fact, we can prove this more directly, at least in characteristics~0 and~2. This argument is not used to prove our main results, but we feel it is worth including nonetheless.

\begin{proposition} \label{prop:simple-etale-crit}
Let $\fT$ be a $k$-linear semi-simple pre-Tannakian category. Suppose that:
\begin{enumerate}
\item The field $k$ has characteristic~0 or~2.
\item The only self-dual simple object of $\fT$ is the monoidal unit $\bone$.
\item If $k$ has characteristic~0 then no object of $\fT$ has dimension $-\tfrac{1}{2}$.
\end{enumerate}
Then every simple algebra in $\fT$ is \'etale.
\end{proposition}

\begin{proof}
By assumption, the non-trivial simples of $\fT$ come in dual pairs; enumerate them as $\{L_i, L_i^*\}_{i \in I}$. Let $A$ be a simple algebra. By Corollary~\ref{cor:Gamma-field}, $A$ contains the trivial representation with multiplicity one, and by Corollary~\ref{cor:simple-Frob}, $A$ is self-dual. We thus have a decomposition
\begin{displaymath}
A = \bone \oplus \bigoplus_{i \in I} \big( L_i^{\oplus m_i} \oplus (L_i^*)^{\oplus m_i} \big),
\end{displaymath}
for some multiplicities $m_i$. Put
\begin{displaymath}
X = \bigoplus_{i \in I} L_i^{\oplus m_i},
\end{displaymath}
so that
\begin{displaymath}
\udim(A) = 1+ 2 \udim(X).
\end{displaymath}
Thus $\udim(A)$ is non-zero; indeed, this is clear if $k$ has characteristic~2, and follows from (c) in characteristic~0. It follows that $A$ is \'etale (\S \ref{ss:etale}(c)).
\end{proof}

\begin{corollary}
If $k$ has characteristic~0 or~2 then every simple algebra in $\uRep(\GG)$ is \'etale.
\end{corollary}

\begin{proof}
We must verify the hypotheses. (b) The dual of $L_{\lambda}$ is $L_{\lambda^{\vee}}$, which is only isomorphic to $L_{\lambda}$ if $\lambda$ is the empty word. (c) The dimension $L_{\lambda}$ is $(-1)^{\ell(\lambda)}$, and so the dimension of every object of $\uRep(\GG)$ is an integer. Thus in characteristic~0, no dimension is equal to $-\tfrac{1}{2}$.
\end{proof}

\begin{remark}
Let $\fT$ be a $k$-linear semi-simple pre-Tannakian category, with $k$ algebraically closed and $\operatorname{char}(k) \ne 2$. Suppose that the second Adams operation $\psi^2$ is trivial (i.e., the identity) on $\rK(\fT)$. Then condition (b) of the proposition is fulfilled, i.e., $\bone$ is the only self-dual simple; see \cite[Proposition~2.1]{delchar}. We note that if $\fT=\uRep(\GG)$ then all Adams operations are trivial on $\rK(\fT)$ \cite[Theorem~8.2]{line}.
\end{remark}

\section{Theorems~\ref*{mainthm} and~\ref*{mainthm3}} \label{s:proof}

In this section we prove the two theorems in the title. We fix a pre-Tannakian category $\fT$ for the duration of \S \ref{s:proof}. Throughout \S \ref{s:proof} we assume that the field $k$ is algebraically closed.

\subsection{The key result}

The following is the key result needed to prove the main theorems.

\begin{theorem} \label{thm:key}
Let $A$ be a simple algebra in $\uRep(\GG) \boxtimes \fT$. Assume that $A$ is \'etale or $\fT$ is semi-simple. Then there exists an open subgroup $U$ of $\GG$ such that $\Res^{\GG}_U(A)$ has a quotient of the form $\bone \boxtimes E$, where $E$ is a simple algebra in $\fT$.
\end{theorem}

Here, and in what follows, we write $\Res^{\GG}_U$ for the functor
\begin{displaymath}
\Res^{\GG}_U \boxtimes \id \colon \uRep(\GG) \boxtimes \fT \to \uRep(U) \boxtimes \fT.
\end{displaymath}
The proof of the theorem will occupy the remainder of \S \ref{s:proof}. For now, we deduce some corollaries, including Theorems~\ref{mainthm} and~\ref{mainthm3}.

\begin{corollary} \label{cor:key}
Let $A$ be a simple algebra in $\uRep(\GG) \boxtimes \fT$. Assume that $A$ is \'etale or $\fT$ is semi-simple. Then there is an injective algebra homomorphism
\begin{displaymath}
A \to \cC(\bR^{(n)}) \boxtimes E
\end{displaymath}
for some $n$ and some simple algebra $E$ in $\fT$. Moreover, if $A$ is \'etale then so is $E$.
\end{corollary}

\begin{proof}
This follows from the theorem and Proposition~\ref{pro:alg-adjunct}. A few things to note. First, the map is injective since $A$ is simple. Second, if $A$ is \'etale then so is $\bone \boxtimes E$, as it is a quotient of $A$, and thus so is $E$ (see \S \ref{ss:etale}). And third, $\GG/U$ is a transitive $\GG$-set, and thus isomorphic to $\bR^{(n)}$ for some $n$.
\end{proof}

The following corollary is Theorem~\ref{mainthm}.

\begin{corollary}
Any simple algebra in $\uRep(\GG)$ is isomorphic to some $\cC(\bR^{(n)})$.
\end{corollary}

\begin{proof}
Let $A$ be a simple algebra in $\uRep(\GG)$. By Corollary~\ref{cor:key}, we see that $A$ is a subalgebra of $\cC(\bR^{(n)})$ for some $n$, and so $A$ is isomorphic to some $\cC(\bR^{(m)})$ (Proposition~\ref{prop:schwartz-sub}).
\end{proof}

The following corollary gives the primary statement in Theorem~\ref{mainthm3}; the secondary statement (giving the precise form of $A$) will be a consequence of Theorem~\ref{mainthm2}, proved in the subsequent section.

\begin{corollary}
Suppose $\fT$ is semi-simple and all simple algebras in $\fT$ are \'etale. Then any simple algebra $A$ in $\uRep(\GG) \boxtimes \fT$ is \'etale.
\end{corollary}

\begin{proof}
$A$ is sub\'etale (Corollary~\ref{cor:key}), and so $A$ is \'etale (Corollary~\ref{cor:simple-sub}).
\end{proof}

\subsection{Notation}

We now introduce some notation that will be used throughout \S \ref{s:proof}. Let $M$ be an object of $\uRep(\GG)^{\boxtimes s} \boxtimes \fT$. We then have a decomposition
\begin{equation} \label{eq:decomp}
M = \bigoplus_{\lambda_1, \ldots, \lambda_s} \big( L_{\lambda_1} \boxtimes \cdots \boxtimes L_{\lambda_s} \boxtimes M_{\lambda_1, \ldots, \lambda_s} \big),
\end{equation}
where each $\lambda_i$ is a weight and $M_{\lambda_1, \ldots, \lambda_s}$ is an object of $\fT$. We define $\ell_i(M)$ to be the maximum value of $\ell(\lambda_i)$ over tuples $(\lambda_1, \ldots, \lambda_s)$ for which $M_{\lambda_1, \ldots, \lambda_s}$ is non-zero. Similarly, we define $\ell_{\rm tot}(M)$ to be the maximum value of $\ell(\lambda_1)+\cdots+\ell(\lambda_s)$ over similar such tuples. When $s=1$, we write $\ell(M)$ for $\ell_1(M)=\ell_{\rm tot}(M)$. We define $T_n(M)$ to be the direct sum of the objects $M_{\lambda_1, \ldots, \lambda_s}$ over tuples $(\lambda_1, \ldots, \lambda_s)$ with $\ell(\lambda_i) \ge n$ for some $i$, and we let $t_n(M)$ be the length of the object $T_n(M)$. We only apply this when $n=\ell_{\rm tot}(M)$. In this case, if $\ell(\lambda_i)=n$ then the other $\lambda_j$'s are empty, and so one can think of $t_n(M)$ as counting the number of length $n$ Delannoy simples that appear in $M$.

Suppose $U$ is an open subgroup of $\GG$. We apply the same definitions above to objects of $\uRep(U) \boxtimes \fT$ via the identification $\uRep(U)=\uRep(\GG)^{\boxtimes s}$. If $M$ is an object of $\uRep(\GG) \boxtimes \fT$ and $M'$ is its restriction to $\uRep(U) \boxtimes \fT$ then the restriction rules for Delannoy show that $\ell_i(M')=\ell(M)$ for all $i$, and $\ell_{\rm tot}(M')=\ell(M)$.

\subsection{The general strategy} \label{ss:step1}

Let $A$ be a simple algebra in $\uRep(\GG) \boxtimes \fT$. We must eventually find an open subgroup $U$ of $\GG$ and a surjective algebra homomorphism $\Res^{\GG}_U(A) \to \bone \boxtimes E$, where $E$ is a simple algebra in $\fT$. Equivalently, we want a surjective algebra homomorphism $\Res^{\GG}_U(A) \to B$, where $B$ is a simple algebra in $\uRep(U) \boxtimes \fT$ with $\ell_{\rm tot}(B)=0$. Rather than finding such a $B$ right away, we first aim to simply find a surjective algebra homomorphism $\Res^{\GG}_U(A) \to B$, where $B$ is smaller than $A$, as measured by $\ell_i$ or $t_n$. By iterating this procedure, we will eventually find a quotient with $\ell_{\rm tot}=0$.

We now execute the first step in this plan. Put $n=\ell(A)$, and assume $n>0$. Let $A'$ be the restriction of $A$ to $\uRep(\GG(0)) \boxtimes \fT$, and decompose $A'$ as in \eqref{eq:decomp}. Let $\fp$, resp.\ $\fq$, be the ideal of $A'$ generated by the summands $\bone \boxtimes L_{\lambda} \boxtimes A'_{\emptyset,\lambda}$, resp.\ $L_{\lambda} \boxtimes \bone \boxtimes A'_{\lambda,\emptyset}$, with $\ell(\lambda)=n$. We have the following simple observation:

\begin{proposition} \label{prop:step1}
We have $\fp\fq=0$. Moreover, we are in one of the following cases:
\begin{enumerate}
\item $\fp+\fq$ is a proper ideal of $A'$.
\item The natural map $A' \to A'/\fp \oplus A'/\fq$ is an isomorphism.
\end{enumerate}
\end{proposition}

\begin{proof}
For subobjects $X$ and $Y$ of $A'$, let $XY$ be the image of $X \otimes Y$ under the multiplication map $A' \otimes A' \to A'$. Let $P$ be the sum of the objects $\bone \boxtimes L_{\lambda} \boxtimes A'_{\emptyset,\lambda}$ with $\ell(\lambda)=n$, so that $\fp = A' P$, and let $Q$ be defined analogously so that $\fq=A'Q$. Now, we have
\begin{displaymath}
P \otimes Q = \bigoplus \big( L_{\lambda} \boxtimes L_{\mu} \boxtimes (A'_{\lambda, \emptyset} \otimes A'_{\emptyset,\mu}) \big),
\end{displaymath}
where the sum is over weights $\lambda$ and $\mu$ of length $n$. Since $\ell_{\rm tot}(A')=n$, we have $A'_{\lambda,\mu}=0$ if $\ell(\lambda)=\ell(\mu)=n$, and so none of the summands above admit a non-zero map to $A'$. Thus $PQ=0$, and so $\fp\fq=0$. (Note that since $\fp$ and $\fq$ are non-zero, this shows that neither $\fp$ nor $\fq$ is the unit ideal.)

Suppose $\fp+\fq$ is the unit ideal. Since $\fp\fq=0$, the Chinese remainder theorem implies that (b) holds. We recall the proof. Multiplying the equation $A'=\fp+\fq$ by $\fp \cap \fq$, we find
\begin{displaymath}
\fp \cap \fq = \fp (\fp \cap \fq) + \fq (\fp \cap \fq) \subset \fp \fq
\end{displaymath}
and so $\fp \cap \fq=0$. Thus $A'=\fp \oplus \fq$, from which (b) follows.
\end{proof}

In case (a), $A'/(\fp+\fq)$ is a non-zero algebra that is smaller than $A$, in the sense that $\ell_i(A'/(\fp+\fq))<n$ for each $i=1,2$. We can take our $B$ to be any simple quotient of this algebra. Thus in this case, finding $B$ is quite easy.

In case (b), more work is required. Before starting on it, we make some general observations. Decompose $A$ as in \eqref{eq:decomp}. We have
\begin{displaymath}
A = \bigoplus_{\ell(\lambda)=n} \big(  L_{\lambda} \boxtimes A_{\lambda} \big) \oplus \cdots
\end{displaymath}
where the remaining terms only use simple Delannoy representations of length $<n$. By the restriction rule for Delannoy, we thus have
\begin{displaymath}
A' = \bigoplus_{\ell(\lambda)=n} \big( (L_{\lambda} \boxtimes \bone \boxtimes A_{\lambda}) \oplus (\bone \boxtimes L_{\lambda} \boxtimes A_{\lambda}) \big) \oplus \cdots
\end{displaymath}
where, again, the omitted terms only use simples of length $<n$. The first terms above all belong to $\fq$, and the second all belong to $\fp$. We thus have
\begin{displaymath}
A'/\fp = \bigoplus_{\ell(\lambda)=n} \big( L_{\lambda} \boxtimes \bone \boxtimes A_{\lambda} \big) \oplus \cdots
\end{displaymath}
and similarly for $A'/\fq$. We thus have
\begin{displaymath}
T_n(A'/\fp)=T_n(A'/\fq)=T_n(A) = \bigoplus_{\ell(\lambda)=n} A_{\lambda}.
\end{displaymath}
We also have $\ell_1(A'/\fp)=n$ and $\ell_2(A'/\fp)<n$, and similarly for $A'/\fq$. Thus, by our various measures, $A'/\fp$ and $A'/\fq$ have the same size as $A$. If we can decompose one of these algebras into a non-trivial direct product, some factor will be smaller than $A$, and we can take $B$ to be a simple quotient of that factor. We establish such a decomposition below.

\begin{remark}
If $A$ is the algebra $\cC(\bR^{(n)})$ in $\uRep(\GG)$ with $n \ge 1$ then we are in case~(a); this is easily seen from the explicit decomposition of $A'$ into simple \'etale algebras, which corresponds to the orbit decomposition of $\GG(0)$ on $\bR^{(n)}$. We expect that we are always in case~(a), but we cannot prove this in general.
\end{remark}

\subsection{The semi-simple case} \label{ss:proof-ss}

Throughout \S \ref{ss:proof-ss}, we assume that $\fT$ is semi-simple. We will show that in case~(b) of Proposition~\ref{prop:step1}, one of the algebras factors. This will yield the following result, which is really what we are after.

\begin{proposition} \label{prop:step2-ss}
Let $A$ be a simple algebra in $\uRep(\GG) \boxtimes \fT$ with $n=\ell(A)$ positive. Then $\Res^{\GG}_{\GG(0)}(A)$ has a simple quotient $B$ such that $\ell_1(B)<n$ or $\ell_2(B)<n$, and $t_n(B) < t_n(A)$.
\end{proposition}

Fix the algebra $A$ for the duration of \S \ref{ss:proof-ss}, and let $A'$ be its restriction to $\GG(0)$. Note that $\Gamma(A)=k$ (Corollary~\ref{cor:Gamma-field}). Let $V$ be the multiplicity space of $L_{\wa} \boxtimes \bone$ in $A$, which is a finite dimensional vector space. Since $A$ is self-dual (Corollary~\ref{cor:simple-Frob}), the multiplicity space of $L_{\wb} \boxtimes \bone$ is naturally identified with the dual space $V^*$. Let $R=\Gamma(A')$ be the invariant algebra of $A'$. This is a finite dimensional commutative $k$-algebra, and decomposes as
\begin{displaymath}
R = V \oplus k \oplus V^*.
\end{displaymath}
by Proposition~\ref{prop:G0-invar}. The following lemma tells us about the structure of $R$.

\begin{lemma} \label{lem:step2-ss}
Suppose $V \ne 0$. We have the following:
\begin{enumerate}
\item The projection $\lambda \colon R \to k$ gives $R$ the structure of a Frobenius algebra.
\item The subspaces $V$ and $V^*$ of $R$ are closed under multiplication.
\item The subspaces $V$ and $V^*$ pair perfectly under the pairing $(,)_{\lambda}$.
\item Neither $V$ nor $V^*$ is contained in the nilradical of $R$.
\item The algebra $R$ has at least three primitive idempotents.
\end{enumerate}
\end{lemma}

\begin{proof}
(a) Let $\mu \colon A \to \bone$ be the projection map in $\uRep(\GG) \boxtimes \fT$, i.e., the unique map that is left-inverse to the unit $\bone \to A$. The pairing $(,)_{\mu}$ on $A$ is perfect by Proposition~\ref{prop:simple-Frob}, and this remains the case for the same pairing on the object $A'$ in $\uRep(\GG(0)) \boxtimes \fT$. Now, $A'$ decomposes as $R \oplus X$, where $X$ is the sum of the non-trivial simples in $A'$. Since $X$ is contained in the kernel of $\mu$, and the multiplication map $R \otimes X \to A'$ lands in $X$, it follows that $R$ and $X$ are orthogonal under $(,)_{\mu}$. Thus $(,)_{\mu}$ restricts to a perfect pairing on $R$. Since $\mu$ restricts to $\lambda$ on $R$, the result follows.

(b) The statement about $V$ follows since $L_{\wa} \otimes L_{\wa}$ does not contain $\bone$ or $L_{\wb}$ as a summand (\S \ref{ss:simple}). The statement about $V^*$ is similar.

(c) If $x \in V$ and $y \in V \oplus k$ then $xy \in V$ by (b), and so $(x,y)_{\lambda}=0$. Thus $V$ is orthogonal to $V \oplus k$, and so must pair perfectly with $V^*$.

(d) Suppose $V$ is contained in the nilradical. Let $x \in V$ and let $y \in V^*$. Then $xy$ decomposes as $a+b+c$, where $a \in V$, $b \in k$, and $c \in V^*$. Since $xy$ and $a$ are nilpotent, so is $b+c$. But $b+c$ lives in $k \oplus V^*$, which is a subalgebra by (b). Since $V^*$ is an ideal of this subalgebra, again by (b), it follows that the map $\lambda \colon k \oplus V^* \to k$ is an algebra homomorphism, and so $b=\lambda(b+c)$ is nilpotent, and thus vanishes. Hence $(x,y)_{\lambda}=b=0$. Since $x$ and $y$ are arbitrary, this shows that $(,)_{\lambda}$ induces the zero pairing between $V$ and $V^*$. This contradicts (c), since $V$ is non-zero.

(e) The nilradical of $V \oplus k$ is contained in $V$; indeed, if $a+b$ is nilpotent then $b=\lambda(a+b)$ vanishes, as in the previous paragraph. We thus see that the quotient of the algebra $V \oplus k$ by its nilradical has dimension at least two by (d). It follows that there is a non-zero idempotent $x \in V$. Similarly, there is a non-zero idempotent $y \in V^*$. Now, $x$ gives a non-trivial decomposition $R=xR \oplus (1-x)R$. Since $y$ is not equal to 0, 1, $x$, or $1-x$, it follows that $xy$ gives a non-trivial decomposition of $xR$, or $(1-x)y$ gives a non-trivial decomposition of $(1-x)R$, or both. Thus $R$ decomposes into a product with at least three non-trivial factors, and so the result follows.
\end{proof}

With the above result in hand, we can now prove the proposition.

\begin{proof}[Proof of Proposition~\ref{prop:step2-ss}]
Let $\fp$ and $\fq$ be the ideals of $A'$ as in \S \ref{ss:step1}. If $\fp+\fq$ is not the unit ideal then we can take $B$ to be any simple quotient of $A'/(\fp+\fq)$. Since $\ell_i(B)<n$ for $i=1,2$, we have $t_n(B)=0$, and so $t_n(B)<t_n(A)$, as required.

Now suppose $\fp+\fq$ is the unit ideal, and so $A'=A'/\fp \oplus A'/\fq$ by Proposition~\ref{prop:step1}. We have
\begin{displaymath}
\Gamma(A') = \Gamma(A'/\fp) \oplus \Gamma(A'/\fq).
\end{displaymath}
Thus $\Gamma(A')$ is at least two dimensional, and so $V \ne 0$. Since $\Gamma(A')$ has at least three primitive idempotents (Lemma~\ref{lem:step2-ss}), it follows that at least one of $\Gamma(A'/\fp)$ or $\Gamma(A'/\fq)$ has a non-trivial idempotent; without loss of generality, suppose that $A'/\fp$ does. Then $A'/\fp$ factors as $B_1 \oplus B_2$, where $B_1$ and $B_2$ are non-zero algebras. We have $T_n(A'/\fp)=T_n(B_1) \oplus T_n(B_2)$, and so one of $t_n(B_1)$ or $t_n(B_2)$ is strictly less than $t_n(A'/\fp)=t_n(A)$; without loss of generality, say $t_n(B_1)<t_n(A)$. We have $\ell_i(B_1) \le \ell_i(A'/\fp)$, and so $\ell_1(B_1) \le n$ and $\ell_2(B_1)<n$. We can therefore take $B$ to be any simple quotient of $B_1$.
\end{proof}

\subsection{The \'etale case} \label{ss:proof-etale}

Throughout \S \ref{ss:proof-etale}, we work in the \'etale case, and do not assume that $\fT$ is semi-simple. Once again, our goal is to factor one of the algebras in case~(b) of Proposition~\ref{prop:step1}. However, this will not be as straightforward as in the semi-simple case: we must restrict to a small open subgroup to obtain the factorization. This yields the following proposition, which is what we are really after.

\begin{proposition} \label{prop:step2-etale}
Let $A$ be a simple \'etale algebra in $\uRep(\GG) \boxtimes \fT$ with $n=\ell(A)$ positive. Then there is an open subgroup $U \cong \GG^s$ of $\GG$ such that the restriction of $A$ to $\uRep(U) \boxtimes \fT$ has a simple quotient algebra $B$ such that $\ell_i(B)<n$ for all but at most one value of $i \in [s]$, and $t_n(B)<t_n(A)$.
\end{proposition}

Before proving this, we require a lemma. In what follows, for an object $M$ of $\uRep(U) \boxtimes \fT$, we write $M^U$ for the object $M_{\emptyset,\ldots,\emptyset}$ in the notation of \eqref{eq:decomp}.

\begin{lemma} \label{lem:etale-quotient}
Let $U \cong \GG^s$ be an open subgroup of $G$, and let $A$ be a simple \'etale algebra in $\uRep(U) \boxtimes \fT$ with $n=\ell_{\rm tot}(A)$ positive and $\ell_i(A)<n$ for all but one $i \in [s]$. Then there is an open subgroup $V\cong \GG^{s+1}$ of $U$ and a simple quotient $B$ of $\Res^U_V(A)$ such that one of the following two cases hold:
\begin{enumerate}
\item We have $\ell_i(B)<n$ for all but at most one $i \in [s+1]$, and $t_n(B)<t_n(A)$.
\item We have $\ell_i(B)<n$ for all but one $i \in [s+1]$, we have $T_n(B)=T_n(A)$, and the natural map $A^U \to B^V$ is a strict inclusion.
\end{enumerate}
\end{lemma}

\begin{proof}
Without loss of generality, suppose $\ell_1(A)=n$ and $\ell_i(A)<n$ for $2 \le i \le s$. We take $V$ to be the open subgroup of $U$ that corresponds to $\GG(0) \times \GG^{s-1}$ under the identification $U \cong \GG^s$. Identify $\uRep(U) \boxtimes \fT$ with $\uRep(\GG) \boxtimes \fT'$, where $\fT'=\uRep(\GG^{s-1}) \boxtimes \fT$, and similarly identify $\uRep(V) \boxtimes \fT$ with $\uRep(\GG(0)) \boxtimes \fT'$.

Let $A'=\Res^U_V(A)$, and let $\fp$ and $\fq$ be the ideals of $A'$ defined in \S \ref{ss:step1}. If we are in case~(a) of Proposition~\ref{prop:step1}, then we can take $B$ to be any simple quotient of $A'/(\fp+\fq)$ and we are in case (a) of the present proposition, with $t_n(B)=0$. Thus suppose we are in case~(b) of Proposition~\ref{prop:step1}, so that $A'=A'/\fp \oplus A'/\fq$. If either $A'/\fp$ or $A'/\fq$ factors as a non-trivial direct product then we can take $B$ to be an appropriate simple quotient of one of the factors (exactly as in the proof of Proposition~\ref{prop:step2-ss}), and we are again in case (a) of the present proposition. Thus suppose that $A'/\fp$ and $A'/\fq$ do not factor, which means they are simple (since they are \'etale).

Decompose $A$ as in \eqref{eq:decomp}. Note that $A_{\emptyset^s}=A^U$ is an \'etale algebra in $\fT$ by Proposition~\ref{prop:et-invar}; here $\emptyset^s$ is the length $s$ tuple $(\emptyset, \ldots, \emptyset)$. It satisfies $\Gamma(A^U)=\Gamma(A)=k$, and is thus a simple algebra. Put $M=A_{\bb,\emptyset^{s-1}}$ and $N=A_{\ww,\emptyset^{s-1}}$. These are $A^U$-modules and dual to each other (as objects of $\fT$) since $A$ is self-dual (being \'etale). We have
\begin{displaymath}
(A'/\fp)^V \oplus (A'/\fq)^V = (A')^V = M \oplus A^U \oplus N
\end{displaymath}
by Proposition~\ref{prop:G0-invar}. This isomorphism holds as $A^U$-modules. Both terms on the left are non-zero, since $A'/\fp$ and $A'/\fq$ are non-zero algebras. Thus the left side has length at least two as an $A^U$-module. Since $A^U$ is simple as a module over itself, it follows that at least one of $M$ or $N$ is non-zero, and so both are (since they are dual to each other). Thus the right side above has length at least three as an $A^U$-module. It follows that one of the terms on the left has length at least two as an $A^U$-module. We take $B$ to be this factor of $A'$. This belongs to case (b) of the present proposition. Indeed, we have already seen that the first two conditions hold (see the discussion following Proposition~\ref{prop:step1}), and the containment $A^U \subset B^V$ is strict by construction. (Note that the map $A^U \to B^V$ is indeed injective, since $A^U$ is a simple algebra.)
\end{proof}

\begin{proof}[Proof of Proposition~\ref{prop:step2-etale}]
Put $B_1=A$ and $U_1=\GG$. Apply Lemma~\ref{lem:etale-quotient} to $B_1$, and let $B_2$ be the resulting simple quotient of the restriction of $B_1$ in $\uRep(U_2) \boxtimes \fT$, where $U_2 \cong \GG^2$. In case~(a), we are done, so suppose we are in case~(b). Apply the lemma again. Let $B_3$ be the resulting simple quotient of the restriction of $B_2$ in $\uRep(U_3) \boxtimes \fT$, where $U_3 \cong \GG^3$. In case~(a), we are done, so suppose we are in case~(b).

Continue in this manner. If we ever reach case~(a) then we are done. Thus suppose this never happens; we will obtain a contradiction. We thus obtain a chain of open subgroups $U_1 \supset U_2 \supset \cdots$ of $\GG$, and simple \'etale algebras $B_i$ in $\uRep(U_i) \boxtimes \fT$ such that:
\begin{itemize}
\item $B_{i+1}$ is a quotient of the restriction of $B_i$,
\item the inclusion $B_i^{U_i} \subset B_{i+1}^{U_{i+1}}$ is strict, and
\item $T_n(B_i)=T_n(A)$ for all $i$.
\end{itemize}
Each $B_i^{U_i}$ is a simple \'etale algebra in $\fT$, using the argument from the proof of Lemma~\ref{lem:etale-quotient}. Moreover, $T_n(B_i)=T_n(A)$ is a module over it. The map $B_i^{U_i} \to \ul{\End}(T_n(A))$ is injective since $B_i^{U_i}$ is simple. However, this gives an upper bound on the length of $B_i^{U_i}$, and so we have a contradiction.
\end{proof}

\subsection{Completion of proof}

We now complete the proof of Theorem~\ref{thm:key}. Suppose $A$ is a simple algebra in $\uRep(\GG) \boxtimes \fT$ with $\ell(A)=n$. We have just seen that (under the hypotheses of Theorem~\ref{thm:key}), the restriction of $A$ to some open subgroup $U \cong \GG^s$ admits a simple quotient in which only one $\GG$ has any length $n$ simples, and the number of such simples is smaller than the corresponding number in $A$ (as measured by $t_n)$. By iterating this construction, we can find a quotient that has no length $n$ simples. This is the content of the following lemma.

\begin{lemma} \label{lem:key-1}
Let $A$ be a simple algebra in $\uRep(\GG) \boxtimes \fT$ with $n=\ell(A)$ positive. Assume $A$ is \'etale or $\fT$ is semi-simple. Then there is an open subgroup $U \cong \GG^s$ of $\GG$ such that the restriction of $A$ to $\uRep(U) \boxtimes \fT$ admits a simple quotient $B$ with $\ell_i(B)<n$ for all $1 \le i \le s$.
\end{lemma}

\begin{proof}
We proceed by induction on $t_n(A)$. By Proposition~\ref{prop:step2-ss} (if $\fT$ is semi-simple) or Proposition~\ref{prop:step2-etale} (if $A$ is \'etale), there is an open subgroup $V \cong \GG^p$ of $\GG$ such that $\Res^{\GG}_V(A)$ admits a simple quotient $C$ satisfying $\ell_i(C)<n$ for all but at most one $i \in [p]$, and $t_n(C)<t_n(A)$. Note that if $\fT$ is semi-simple then we actually get $p=2$ here. For notational simplicity, suppose $\ell_i(C)<n$ for all $2 \le i \le p$. If $\ell_1(C)<n$ then we can simply take $U=V$ and $B=C$. We thus assume $\ell_1(C)=n$ in what follows.

Identify $\uRep(V) \boxtimes \fT$ with $\uRep(\GG) \boxtimes \fT'$, where $\fT'=\uRep(\GG)^{\boxtimes (p-1)} \boxtimes \fT$. Let $C'$ be $C$, regarded in $\uRep(\GG) \boxtimes \fT'$. We have $\ell(C)=n$. Also,  $T_n(C')=\bone^{\boxtimes (p-1)} \boxtimes T_n(C)$, and so $t_n(C')=t_n(C)$. Thus, by induction, there is an open subgroup $U_0 \cong \GG^q$ of $\GG$ such that $\Res^{\GG}_{U_0}(C')$ admits a simple quotient $B'$ with $\ell_i(B')<n$ for all $1 \le i \le q$. Now, under the isomorphism $V \cong \GG^p$, there is an open subgroup $U$ of $V$ that corresponds to $U_0 \times \GG^{p-1}$; note that $U \cong \GG^{p+q-1}=\GG^s$. We have an identification $\uRep(U) \boxtimes \fT \cong \uRep(U_0) \boxtimes \fT'$. Let $B$ be the algebra of $\uRep(U) \boxtimes \fT$ corresponding to $B'$. Then $B$ is a quotient of $\Res^{\GG}_U(A)$, and $\ell_i(B)<n$ for all $1 \le i \le s$, as required.
\end{proof}

By iterating the above lemma, we eventually obtain a simple quotient of a restriction of $A$ in which no Delannoy simples except the trivial representation appear, as we now explain:

\begin{proof}[Proof of Theorem~\ref{thm:key}]
We proceed by induction on $\ell(A)$. If $\ell(A)=0$, there is nothing to prove. Suppose now that $n=\ell(A)$ is positive. By Lemma~\ref{lem:key-1}, there is an open subgroup $V \cong \GG^s$ of $\GG$ such that $\Res^{\GG}_V(A)$ admits a simple quotient $B$ with $\ell_i(B)<n$ for all $1 \le i \le s$. Concretely, this means for each copy of $\GG$, only simples of length $<n$ appear in $B$. Note that if $A$ is \'etale then so is $B$, since a quotient of an \'etale algebra is \'etale.

Put $\fT_i=\uRep(\GG)^{\boxtimes i} \boxtimes \fT$. Regard $B=B_s$ as an object of $\uRep(\GG) \boxtimes \fT_{s-1}$. We have $\ell(B)<n$. By induction, there is an open subgroup $U_1$ of $\GG$ and a quotient of $\Res^{\GG}_{U_1}(B_s)$ of the form $\bone \boxtimes B_{s-1}$, where $B_{s-1}$ is a simple algebra of $\fT_{s-1}$. Again, if $B_s$ is \'etale then so is $B_{s-1}$.

Regard $B_{s-1}$ as an object of $\uRep(\GG) \boxtimes \fT_{s-2}$. Since $\bone \boxtimes B_{s-1}$ is a quotient of $\Res^{\GG}_{U_1}(B)$, any Delannoy simple appearing in $B_{s-1}$ (for any copy of $\GG$) has length $<n$. Thus $\ell(B_{s-1})<n$. As above, there is an open subgroup $U_2$ of $\GG$ such that $\Res^{\GG}_{U_2}(B_{s-1})$ has a quotient of the form $\bone \boxtimes B_{s-2}$, where $B_{s-2}$ is a simple algebra of $\fT_{s-2}$.

Continuing in this manner, we obtain open subgroups $U_1, \ldots, U_s$ of $\GG$ such that the restriction of $B$ from $\GG^s$ to $U_1 \times \cdots \times U_s$ admits a quotient of the form $\bone^{\boxtimes s} \boxtimes E$, where $E$ is a simple algebra of $\fT$. Under the identification $V \cong \GG^s$, there is an open subgroup $U$ of $V$ that corresponds to $U_1 \times \cdots \times U_s$. We thus see that $\Res^{\GG}_U(B)$, regarded as an object of $\uRep(U) \boxtimes \fT$, has $\bone \boxtimes E$ as a quotient, as desired.
\end{proof}

\section{Theorem~\ref*{mainthm2}} \label{s:proof2}

\subsection{Overview}

Let $\fT$ be a pre-Tannakian category and let $A$ be a simple \'etale algebra in $\uRep(\GG) \boxtimes \fT$. We must show that $A$ is isomorphic to $\cC(\bR^{(n)}) \boxtimes E$ for some $n$ and some simple \'etale algebra $E$ of $\fT$; this is what Theorem~\ref{mainthm2} asserts. So far, we have seen that $A$ is a subalgebra of such an algebra (Corollary~\ref{cor:key}). Now, the theory developed in \cite{pregalois, discrete} shows that the category of \'etale algebras in a pre-Tannakian category is equivalent to a category of the form $\bS(G)$ for some pro-oligomorphic group $G$. This allows us to translate our problem into one about $G$-sets, which is more elementary; the proof of Proposition~\ref{prop:schwartz-sub} is another example of this method. Ultimately, we deduce what we want from a special property of the group $\GG$, namely, that it is split (\S \ref{ss:split}).

In this section, we first discuss split oligomorphic groups. We then recall the material from \cite{pregalois} and \cite{discrete}. Finally, we apply all of this in the case of interest to prove Theorem~\ref{mainthm2}.

\subsection{Split oligomorphic groups} \label{ss:split}

Theorem~\ref{mainthm2} states that any simple \'etale algebra in $\uRep(\GG) \boxtimes \fT$ factors into a tensor product. The following definition gives a combinatorial analog of this property.

\begin{definition}
A pro-oligomorphic group $G$ is \defn{split} if for any pro-oligomorphic group $H$, any transitive $G \times H$ set is isomorphic to one of the form $X \times Y$, where $X$ is a transitive $G$-set and $Y$ is a transitive $H$-set.
\end{definition}

The following proposition gives several equivalent conditions for being split.

\begin{proposition} \label{prop:split}
Let $G$ be a pro-oligomorphic group. The following are equivalent:
\begin{enumerate}
\item $G$ is split.
\item For any pro-oligomorphic group $H$, any open subgroup of $G \times H$ is of the form $U \times V$, where $U$ is an open subgroup of $G$ and $V$ is an open subgroup of $H$.
\item Every open subgroup of $G$ is its own normalizer.
\item The automorphism group of any transitive $G$-set is trivial.
\item Whenever $V \subset U$ is a finite index containment of open subgroups, we have $U=V$.
\end{enumerate}
Moreover, if $G$ is split then any open subgroup of $G$ is split.
\end{proposition}

\begin{proof}
(b) $\Rightarrow$ (a) is clear.

(a) $\Rightarrow$ (b) If $W$ is an open subgroup of $G \times H$, then $(G \times H)/W$ is isomorphic to $(G/U') \times (H/V')$ for some $U' \subset G$ and $V' \subset H$, and so $W$ is conjugate in $G \times H$ to $U' \times V'$, and thus of the form $U \times V$.

(c) $\Rightarrow$ (b) Let $H$ be a pro-oligomorphic group and let $W$ be an open subgroup of $G \times H$. Let $U$ be the inverse image of $W$ under the inclusion $G \to G \times H$, and similarly define $V \subset H$. We claim that the containment $U \times V \subset W$ is an equality. Let $(g,h) \in W$ be given. For any $u \in U$, the element
\begin{displaymath}
(g,h)(u,1)(g,h)^{-1} = (gug^{-1}, 1)
\end{displaymath}
belongs to $W$, and so $gug^{-1} \in U$. We thus find $gUg^{-1} \subset U$. Applying the same reasoning to $(g^{-1}, h^{-1}) \in W$, we find\footnote{We thank the referee for pointing out this simplification.} $g^{-1}Ug \subset U$. Thus $gUg^{-1}=U$, and so $g$ belongs to the normalizer of $U$, which is $U$ by our assumption. Since $(g,1) \in W$, it follows that $(1,h) \in W$, and so $h \in V$. Thus the claim holds, and so (b) holds.

(b) $\Rightarrow$ (c) Let $U$ be an open subgroup of $G$. Consider the group homomorphism
\begin{displaymath}
\rN_G(U) \times \rN_G(U) \to \rN_G(U)/U  \times \rN_G(U)/U.
\end{displaymath}
Let $W$ be the inverse image of the diagonal subgroup of the target. This contains $U \times U$, and is thus open. By assumption, we have $W=W_1 \times W_2$ for open subgroups $W_i \subset G$. Clearly, we have
\begin{displaymath}
W_1 \times 1 = W \cap (G \times 1)=U \times 1, \qquad
1 \times W_2 = W \cap (1 \times G)=1 \times U.
\end{displaymath}
Thus $W=U \times U$, and so $\rN_G(U)=U$, as required.

(c) $\Leftrightarrow$ (d) If $X=G/U$ is a transitive $G$-set then the automorphism group of $X$ is $\rN_G(U)/U$, and so $\Aut(X)$ is trivial if and only if $U=\rN_G(U)$.

(e) $\Rightarrow$ (c) If $U$ is an open subgroup then $U$ has finite index in $\rN_G(U)$, as $U\backslash \rN_G(U)/U = \rN_G(U)/U$ is finite. Thus $U=\rN_G(U)$.

(c) $\Rightarrow$ (e) Let $V \subset U$ be a finite index containment. Let $W$ be the intersection of the groups $gVg^{-1}$ with $g \in U$. This is a finite intersection, and so $W$ is an open subgroup of finite index in $U$. Since $W$ is normal in $U$, we have $U \subset \rN_G(W)$, and so $U \subset W$ by our assumption. Thus $V=U$, as required.

Finally, if $G$ satisfies (e) then clearly any open subgroup of $G$ does as well.
\end{proof}

\begin{remark}
Let $\cF$ be a Fra\"iss\'e class with limit $\Omega$, and let $G=\Aut(\Omega)$, which acts oligomorphically on $\Omega$; see \cite[\S 6.2]{repst} for background. If $X$ is a member of $\cF$ then the set $\Omega^{[X]}$ of embeddings $X \to \Omega$ is a transitive $G$-set, and its automorphism group contains $\Aut(X)$ (with equality if $X$ is definably closed). Thus a necessary condition for $G$ to be split is that all members of $\cF$ have trivial automorphism group. We do not know if this is a sufficient condition.
\end{remark}

\subsection{A property of split groups}

Let $G$ and $K$ be pro-oligomorphic groups, and let $\Phi \colon \bS(G) \to \bS(K)$ be a functor. We make the following definitions:
\begin{itemize}
\item $\Phi$ is \defn{left exact} if it commutes with finite limits.
\item $\Phi$ is \defn{right exact} if it commutes with finite co-limits.
\item $\Phi$ is \defn{exact} if it is both left and right exact.
\item $\Phi$ is \defn{additive} if it commutes with finite co-products.
\item $\Phi$ is \defn{atomic} if it maps transitive sets to transitive sets.
\end{itemize}
If $\Phi$ is left exact and additive then the following conditions are equivalent: (i) $\Phi$ is exact; (ii) $\Phi$ preserves quotients by equivalence relations; and (iii) $\Phi$ preserves surjections. If $\Phi$ is additive and atomic then it preserves surjections. If $\Phi$ is an exact atomic functor then it is fully faithful (look at graphs of morphisms), and so $\bS(G)$ can be identified with a full subcategory of $\bS(K)$ that is closed under finite (co-)limits.

We require the following property of split groups.

\begin{proposition} \label{prop:split-quot}
Let $G$, $H$, and $K$ be pro-oligomorphic groups, with $G$ split, and let
\begin{displaymath}
\Phi \colon \bS(G) \to \bS(K), \qquad \Psi \colon \bS(H) \to \bS(K)
\end{displaymath}
be exact functors. Suppose the following condition holds:
\begin{enumerate}
\item[($\ast$)] $\Phi(X) \times \Psi(Y)$ is transitive, whenever $X \in \bS(G)$ and $Y \in \bS(H)$ are transitive.
\end{enumerate}
Let $X$ be a transitive $G$-set, let $Y$ be a transitive $H$-set, and let $Z$ be a quotient of the $K$-set $\Phi(X) \times \Psi(Y)$. Then there exist quotients $X'$ of $X$ and $Y'$ of $Y$ such that $Z \cong \Phi(X') \times \Psi(Y')$.
\end{proposition}

We require some lemmas before proving the proposition.

\begin{lemma} \label{lem:split-quot-1}
In the setting of Proposition~\ref{prop:split-quot}, there exists an exact atomic functor
\begin{displaymath}
\Theta \colon \bS(G \times H) \to \bS(K)
\end{displaymath}
with a functorial isomorphism
\begin{displaymath}
\Theta(X \times Y) = \Phi(X) \times \Psi(Y)
\end{displaymath}
for $X \in \bS(G)$ and $Y \in \bS(H)$.
\end{lemma}

\begin{proof}
For a pro-oligomorphic group $L$, let $\bT(L)$ be the full subcategory of $\bS(L)$ spanned by the transitive $L$-sets. Since $G$ is split, we have an equivalence
\begin{displaymath}
\bT(G) \times \bT(H) \to \bT(G \times H), \qquad (X, Y) \mapsto X \times Y.
\end{displaymath}
We therefore have a functor
\begin{displaymath}
\Theta_0 \colon \bT(G \times H) \to \bS(K), \qquad \Theta_0(X \times Y) = \Phi(X) \times \Psi(Y),
\end{displaymath}
which admits a unique additive extension $\Theta$ to all of $\bS(G \times H)$. We clearly have a functorial isomorphism as stated, as one sees by decomposing $X$ and $Y$ into orbits, and it follows from $(\ast)$ that $\Theta$ is atomic. It remains to show that $\Theta$ is exact.

Suppose we have maps of $G$-sets $X_1, X_2 \to X$ and maps of $H$-sets $Y_1,Y_2 \to Y$. Then we have a natural isomorphism
\begin{displaymath}
(X_1 \times Y_1) \times_{X \times Y} (X_2 \times Y_2) = (X_1 \times_X X_2) \times (Y_1 \times_Y Y_2).
\end{displaymath}
Applying $\Theta$, we obtain natural isomorphisms
\begin{align*}
\Theta((X_1 \times Y_1) \times_{X \times Y} (X_2 \times Y_2))
&= \Phi(X_1 \times_X X_2) \times \Psi(Y_1 \times_Y Y_2) \\
&= (\Phi(X_1) \times_{\Phi(X)} \Phi(X_2)) \times (\Psi(Y_1) \times_{\Psi(Y)} \Psi(Y_2)) \\
&= (\Phi(X_1) \times \Psi(Y_1)) \times_{\Phi(X) \times \Psi(Y)} (\Phi(X_2) \times \Psi(Y_2)) \\
&= \Theta(X_1 \times Y_1) \times_{\Theta(X \times Y)} \Theta(X_2 \times Y_2),
\end{align*}
where we have made use of the left exactness of $\Phi$ and $\Psi$. We thus see that $\Theta$ preserves fiber products, at least for $(G \times H)$-sets that are products. The general case follows from decomposing into orbits (which are always products), and using the fact that $\Theta$ is additive. Since $\Theta$ also preserves final objects, it is left exact. Finally, since $\Theta$ is atomic, it preserves surjections, and is thus exact.
\end{proof}

\begin{lemma} \label{lem:split-quot-2}
Let $L$ and $K$ be pro-oligomorphic groups, let $\Theta \colon \bS(L) \to \bS(K)$ be an exact atomic functor, and let $X$ be an $L$-set. Then any quotient of $\Theta(X)$ has the form $\Theta(Y)$, for some quotient $Y$ of $X$.
\end{lemma}

\begin{proof}
A quotient of $\Theta(X)$ has the form $\Theta(X)/R'$, where $R'$ is an equivalence relation on $\Theta(X)$. Since $\Theta$ is atomic, it induces a bijection between $L$-subsets of $X \times X$ and $K$-subsets of $\Theta(X \times X) = \Theta(X) \times \Theta(X)$. Thus $R'$ has the form $\Theta(R)$, for a unique $L$-subset $R$ of $X \times X$. One readily verifies that $R$ is an equivalence relation on $X$. Since $\Theta$ is exact, we have $\Theta(X/R)=\Theta(X)/\Theta(R)$, as required.
\end{proof}

We are now ready to prove the proposition.

\begin{proof}[Proof of Proposition~\ref{prop:split-quot}]
Let  $\Theta \colon \bS(G \times H) \to \bS(K)$ be the exact atomic functor provided by Lemma~\ref{lem:split-quot-1}. Then $Z$ is a quotient of $\Theta(X \times Y)$. By Lemma~\ref{lem:split-quot-2}, we see that $Z=\Theta(Z')$, where $Z'$ is a quotient of $X \times Y$ as a $(G \times H)$-set. Since $G$ is split, $Z'$ has the form $X' \times Y'$, where $X'$ is a quotient of $X$ and $Y'$ is a quotient of $Y$. We thus have
\begin{displaymath}
Z = \Theta(X' \times Y') = \Phi(X') \times \Psi(Y'),
\end{displaymath}
as required.
\end{proof}

\subsection{The oligomorphic fundamental group} \label{ss:oligo-fund}

Let $\fT$ be a pre-Tannakian category. We assume that $k$ is algebraically closed for the remainder of the paper. Let $\Et(\fT)$ be the category of \'etale algebras in $\fT$ (with morphisms being algebra homomorphisms). By \cite[Theorem~6.1]{discrete}, $\Et(\fT)^{\op}$ is a pre-Galois category, and so by the main theorem of \cite{pregalois}, it has the form $\bS(G)$ for some pro-oligomorphic group $G$.

\begin{definition}
The \defn{oligomorphic fundamental group} of $\fT$ is any pro-oligomorphic group $G$ such that $\bS(G)$ is equivalent to $\Et(\fT)^{\op}$.
\end{definition}

\begin{remark}
The group $G$ above is not uniquely defined. For instance, let $\fS^{(1)}$ be the group of all permutations of a countably infinite set, let $\fS^{(2)}$ be the subgroup consisting of finitary permutations, and let $\fS^{(3)}$ be the group of all permutations of an uncountable set. These three groups are pairwise non-isomorphic, but the three categories $\bS(\fS^{(i)})$ are all equivalent. This non-uniqueness will not cause an issue in what follows.
\end{remark}

\begin{remark}
Suppose $\fT$ has the form $\uRep(G, \mu)$ for some pro-oligomorphic group $G$ with measure $\mu$. Then every $G$-set $X$ gives an \'etale algebra $\cC(X)$ in $\fT$, and so there is a functor $\bS(G) \to \Et(\fT)^{\op}$. This functor is often an equivalence, e.g., it is for the Delannoy category by Theorem~\ref{mainthm}. However, it can fail to be an equivalence. An example is provided by the infinite general linear group over a finite field equipped with the parabolic topology; see \cite{interp} for details.
\end{remark}

Fix $\fT$ and let $G$ be its oligomorphic fundamental group. For a $G$-set $X$, we write $\cA(X)$ for the corresponding \'etale algebra of $\fT$. The functor $X \mapsto \cA(X)$ is contravariant and fully faithful, by definition. A map $Y \to X$ of $G$-sets is surjective (resp.\ injective) if and only if the corresponding map $\cA(X) \to \cA(Y)$ is injective (resp.\ surjective) in the abelian category $\fT$ (i.e., has vanishing kernel or co-kernel). For $G$-sets $X$ and $Y$, we have natural algebra isomorphisms
\begin{displaymath}
\cA(X) \oplus \cA(Y) = \cA(X \amalg Y), \qquad
\cA(X) \otimes \cA(Y) = \cA(X \times Y).
\end{displaymath}
The algebra $\cA(X)$ is simple if and only if the $G$-set $X$ is transitive. More generally, the invariant algebra of $\cA(X)$ has the form $k^n$, where $n$ is the number of $G$-orbits on $X$. These statements all follow from the construction and the basic properties of \'etale algebras in \S \ref{ss:etale}.

\subsection{Back to Delannoy}

We now return to the Delannoy category $\uRep(\GG)$. We begin with a group-theoretic observation.

\begin{proposition} \label{prop:del-split}
The group $\GG=\Aut(\bR, <)$ is split in the sense of \S \ref{ss:split}.
\end{proposition}

\begin{proof}
Consider the point $x=(1,\ldots,n)$ in $\bR^{(n)}$. The stabilizer of $x$ is $\GG(A)$, where $A=\{1,\ldots,n\}$. The point $x$ is the only fixed point of $\GG(A)$ on $\bR^{(n)}$. It follows that any $\GG$-map $\bR^{(n)} \to \bR^{(n)}$ maps $x$ to itself, and is therefore the identity; hence the automorphism group of $\bR^{(n)}$ is trivial. Since these account for all transitive $\GG$-sets, it follows that $\GG$ is split (Proposition~\ref{prop:split}).
\end{proof}

We next translate the above result to an analogous result about \'etale algebras. Let $\fT$ be a pre-Tannakian category.

\begin{proposition} \label{prop:split-etale-alg}
Let $A$ be a simple \'etale algebra in $\uRep(\GG) \boxtimes \fT$. Suppose there is an algebra homomorphism $ A \to \cC(\bR^{(n)}) \boxtimes E$, where $E$ is a simple \'etale algebra in $\fT$. Then $A$ is isomorphic to $\cC(\bR^{(m)}) \boxtimes E'$ for some $m \le n$ and simple \'etale subalgebra $E'$ of $E$.
\end{proposition}

\begin{proof}
Let $H$ and $K$ be the oligomorphic fundamental groups of $\fT$ and $\uRep(\GG) \boxtimes \fT$. Note that $\GG$ is the oligomorphic fundamental group of $\uRep(\GG)$ by Theorem~\ref{mainthm}; for a $\GG$-set $X$, we write $\cC(X)$ in place of $\cA(X)$. If $E$ is an \'etale algebra in $\fT$ then $\bone \boxtimes E$ is an \'etale algebra in $\uRep(\GG) \boxtimes \fT$. Moreover, the functor $E \mapsto \bone \boxtimes E$ preserves injections, surjections, and simple algebras. It follows that we have an exact atomic functor
\begin{displaymath}
\Psi \colon \bS(H) \to \bS(K)
\end{displaymath}
satisfying $\cA(\Psi(Y)) = \bone \boxtimes \cA(Y)$. Similarly, we have an exact atomic functor
\begin{displaymath}
\Phi \colon \bS(\GG) \to \bS(K)
\end{displaymath}
satisfying $\cA(\Phi(X)) = \cC(X) \boxtimes \bone$. If $E$ is a simple \'etale algebra in $\fT$ then $\cC(\bR^{(n)}) \boxtimes E$ is a simple \'etale algebra in $\uRep(\GG) \boxtimes \fT$ by \S \ref{ss:etale}(d). It follows that if $X$ is a transitive $\GG$-set and $Y$ is a transitive $H$-set then $\Phi(X) \times \Psi(Y)$ is a transitive $K$-set.

Put $X=\bR^{(n)}$, let $Y$ be the transitive $H$-set with $E=\cA(Y)$, and let $Z$ be the transitive $K$-set with $A=\cA(Z)$. The injection $A \to \cC(\bR^{(n)}) \boxtimes E$ translates to a surjection $\Phi(X) \times \Psi(Y) \to Z$ of $K$-sets. By Proposition~\ref{prop:split-quot}, we have $Z=\Phi(X') \times \Psi(Y')$, where $X'$ is a quotient of $X$ and $Y'$ is a quotient of $Y$. Thus $A=\cC(X') \boxtimes E'$, where $E'=\cA(Y')$. Note that $X'$ has the form $\bR^{(m)}$ for some $m \le n$, by the classification of transitive $\GG$-sets. The result thus follows.
\end{proof}

Theorem~\ref{mainthm2} now follows easily. Indeed, let $A$ be a simple \'etale algebra in $\uRep(\GG) \boxtimes \fT$. By Corollary~\ref{cor:key}, $A$ is a subalgebra of some $\cC(\bR^{(n)}) \boxtimes E$, where $E$ is a simple \'etale algebra in $\fT$. Thus, by Proposition~\ref{prop:split-etale-alg}, $A$ is isomorphic to $\cC(\bR^{(m)}) \boxtimes E'$, as required.

\end{document}